\documentclass[cmbright]{foo}
\usepackage{amsmath,enumerate}
\usepackage{amssymb}
\usepackage{graphicx} 
\usepackage{moreverb}
\usepackage[T1]{fontenc}
\usepackage{graphicx}
\usepackage{psfrag}
\usepackage{fancyhdr}
\usepackage{url}
\usepackage{color}

\def\TT{\mathcal T}
\def\norm#1#2{\|#1\|_{#2}}%

\def\volumeyear{2014}
\def\volumenumber{00}

\def\BbbR{{\mathbb R}}
\def\BbbN{{\mathbb N}}

\newcommand\bi{{\mathbf i}}

\def\supp{\operatorname*{supp}}

\def\IV{{J_h^{V_h}}}
\def\IM{{J_h^{M_h}}}
\def\Ihk{{J_h^k}}

\newcommand\diam[1]{\mbox{\rm diam}\,{#1}}

\newcommand\cutoff{\chi_{S_h}}

\newenvironment{numberedproof}[1]{{\bf Proof of #1:}}{{}\hfill{\hbox{$\Box$}}\par\bigskip}
\newenvironment{proof}{{\bf Proof:}}{{}\hfill{\hbox{$\Box$}}\par\bigskip}

\newcommand{\eex}{\hbox{}\hfill\rule{0.8ex}{0.8ex}}
\newcommand{\eremk}{\eex}
\newcommand{\slp}{\operatorname*{\widetilde V}}
\newcommand{\dlp}{\operatorname*{\widetilde K}}
\newcommand{\slo}{\operatorname*{V}}
\newcommand{\dlo}{\operatorname*{K}}
\newcommand{\adlo}{\operatorname*{K^\prime}}
\newcommand{\hyp}{\operatorname*{D}}
\newcommand{\opA}{\mathfrak A}
\newcommand{\intr}{{\mathrm{int}}}
\newcommand{\extr}{{\mathrm{ext}}}
\newcommand{\dual}{{\mathrm{dual}}}
\newcommand{\bidual}{{\mathrm{bidual}}}

\newtheorem{theorem}{Theorem}[section]
\newtheorem{assumption}[theorem]{Assumption}
\newtheorem{lemma}[theorem]{Lemma}

\newtheorem{remark}[theorem]{Remark}
\newtheorem{corollary}[theorem]{Corollary}

\title{Simultaneous quasi-optimal convergence rates in FEM-BEM coupling}

\begin{document}

\pagestyle{fancy}

\author{J.M.~Melenk,\footnotemark[2]
D. Praetorius,\footnotemark[2] B. Wohlmuth\footnotemark[3]}

\footnotetext[2]
{
Technische Universit\"at Wien, 
Institut f\"ur Analysis und Scientific Computing, 
Wiedner Hauptstr. 8-10, A-1040 Wien, 
melenk@tuwien.ac.at, dirk.praetorius@tuwien.ac.at 
}
\footnotetext[3]
{Technische Universit\"at M\"unchen,
M2 Zentrum Mathematik, 
Boltzmannstra{\ss}e 3,  D-85748 Garching,  
(wohlmuth@ma.tum.de)} 


\begin{abstract}
We consider the symmetric FEM-BEM coupling that connects two linear elliptic second order
partial differential equations posed in a bounded domain $\Omega$ and its complement,
where the exterior problem is restated as an integral equation on the
coupling boundary $\Gamma=\partial\Omega$.
 Under the assumption  that the corresponding 
transmission problem admits a shift theorem for data in $H^{-1+s}$, $s \in [0,s_0]$, $s_0 > 1/2$, we analyze 
the discretization by piecewise polynomials of degree $k$ for the domain variable 
and piecewise polynomials of degree $k-1$ for the flux variable on the coupling boundary. 
Given sufficient regularity we show that (up to logarithmic factors)  
the optimal convergence $O(h^{k+1/2})$ 
in the $H^{-1/2}(\Gamma)$-norm is obtained for the flux variable,
while classical arguments by C\'ea-type quasi-optimality and standard
approximation results provide only $O(h^k)$ for the overall error in the natural
product norm on $H^1(\Omega)\times H^{-1/2}(\Gamma)$.
\end{abstract}

\MOS{65N30; 65N38}


\maketitle
\thispagestyle{empty}

\section{Introduction.}
\label{sec:intro}

The coupling of a linear differential equation in an exterior domain with an equation 
in a domain $\Omega$ arises frequently in numerical computations. One way to tackle 
such a problem is to use a FEM-BEM (finite element--boundary element) coupling
procedure. Several techniques are available 
(see, e.g., \cite{aurada-feischl-fuehrer-karkulik-melenk-praetorius13} for an overview); 
typically, they involve the introduction of an extra unknown $\varphi$ on the 
coupling boundary $\Gamma := \partial\Omega$. Starting with the recent
contribution by Sayas~\cite{sayas2013}, the
stability and convergence analysis of the coupled 
problem is now fairly well-understood: for many coupling procedures, one has quasi-optimality 
in (natural) norms that involve {\em both} the primal variable $u$ on $\Omega$ and the 
variable $\varphi$ on $\Gamma$, 
see again \cite{aurada-feischl-fuehrer-karkulik-melenk-praetorius13}. In many situations, 
the best approximation properties of the spaces used to approximate $u$ and $\varphi$ 
do not match. An example for such a mismatch is the natural setting also used in the present paper: 
here, $u$ is approximated in $H^1(\Omega)$ by piecewise polynomials of degree $k$,
and $\varphi$ is approximated in $H^{-1/2}(\Gamma)$ by piecewise polynomials of degree $k-1$.
The optimal rates are therefore $O(h^k)$ and $O(h^{k+1/2})$, respectively. 
However, the standard
convergence theory based on the C\'ea-type quasi-optimality provides
only the following result:

\begin{quote}
{\bf Proposition (standard quasi-optimality).} If the sought solution $(u,\varphi)$ is smooth, 
then the FEM-BEM approximation $(u_h,\varphi_h)$ satisfies
\begin{align*}
 \norm{u-u_h}{H^1(\Omega)} + \norm{\varphi-\varphi_h}{H^{-1/2}(\Gamma)}
 \le C\,\inf_{v_h,\psi_h}\big(
 \norm{u-v_h}{H^1(\Omega)} + \norm{\varphi-\psi_h}{H^{-1/2}(\Gamma)}
 \big)
 =O(h^k).
\end{align*}
\end{quote}

\noindent
With a refined analysis, we show in the present paper 
that (up to logarithmic terms) convergence $O(h^{k+1/2})$ is obtained for the approximation
of $\varphi$ under suitable assumptions on the regularity of $u$ and $\varphi$ 
and on the boundary value problem; 
see Theorem~\ref{thm:approximation-of-varphi} for the precise statement.

\begin{quote}
{\bf Theorem.} If the sought solution $(u,\varphi)$ is smooth
and if the geometry and coefficients of the differential operator admit a suitable shift theorem, then 
the FEM-BEM approximation $(u_h,\varphi_h)$ satisfies
\begin{align*}
 \norm{u-u_h}{H^1(\Omega)} = O(h^k)
 \quad\text{and}\quad
 \norm{\varphi-\varphi_h}{H^{-1/2}(\Gamma)}
 = O(h^{k+1/2}(1+|\ln h|)),
\end{align*}
where the logarithmic term may even be dropped for $k\ge2$.
\end{quote}

\noindent
We will show this for the symmetric 
FEM-BEM coupling procedure which has independently been proposed by
Costabel~\cite{costabel88a} and Han~\cite{han90}. We believe, however, that 
an extension to other techniques such as the Johnson-N{\'e}d{\'e}lec coupling \cite{johnson-nedelec80} 
is possible. Overall, this gives a first mathematical answer to observations 
in~\cite{afp12,afkp12}, where (optimal) higher-order convergence of 
$\varphi$ was noted for adaptive FEM-BEM computations.

The present work is closely related to our previous works
\cite{melenk-wohlmuth12,melenk-rezaijafari-wohlmuth14,melenk-wohlmuth14a}, where 
the convergence of the Lagrange multiplier in mortar methods \cite{melenk-wohlmuth12}
and the convergence of surface fluxes \cite{melenk-rezaijafari-wohlmuth14} in mixed methods 
were studied. 
The unifying theme of these works is to obtain improved and even optimal convergence 
rates for the quantities associated with lower-dimensional manifolds. This entails a second link 
between all these works: they rely on the same analytical tools, namely, duality arguments 
that require the analysis of elliptic problems with right-hand sides that are supported by a 
thin neighborhood of some lower-dimensional manifold. The basic mechanism that allows 
us to exploit, in a quantifiable way, that the support of the right-hand sides is small, 
is the same one in the works  
\cite{melenk-wohlmuth12,melenk-rezaijafari-wohlmuth14,melenk-wohlmuth14a} and the present one. 
We mention that \cite{melenk-wohlmuth14a} uses a slightly different approach compared to \cite{melenk-wohlmuth12} and provides stronger results. 
The present work is closest to \cite{melenk-wohlmuth14a}.  

We close this introduction with some remarks on the techniques employed. 
As in our previous work, we employ regularity assertions in Besov spaces. A feature of
this approach is that it allows for a rigorous formulation of the regularity properties 
of relevant dual problems and permits us to separate the question of elliptic regularity from
FEM duality arguments as much as possible. 
Nevertheless, the use of Besov spaces is not 
essential and alternative approaches purely based on weighted Sobolev spaces are possible; 
we mention here \cite{apel-pfefferer-roesch12,apel-pfefferer-roesch14} and \cite{larson-massing14} 
as well as \cite{gastaldi-nochetto89} in the context of 
mixed methods. 
Another alternative approach opens up when changing the regularity requirements 
of the solution: in the present paper, we require the solution $u$ to be in 
the Besov space $B^{k+3/2}_{2,1}(\Omega)$; if instead $W^{k+1,\infty}(\Omega)$-regularity
is assumed, then techniques from $L^\infty$-estimates in FEM could be applied.  
We refer to \cite{gastaldi-nochetto89,wang89} for examples in this direction in the context of mixed methods. 

The paper is structured as follows: In Section~\ref{sec:prelim}, we introduce the 
model problem \eqref{eq:strongform}. The variational formulation with the symmetric
coupling is given in Section~\ref{section:weakform}. Our numerical analysis will
rely on duality arguments. The regularity theory for these dual problems, which turn out
to be classical transmission problems, is the topic of Sections~\ref{sec:dual} and \ref{sec:bidual}.
Section~\ref{sec:numerical-analysis} is devoted to the numerical analysis. The main
result is Theorem~\ref{thm:approximation-of-varphi}: 
Estimate~\eqref{eq:thm2} gives an error bound for the variable $u$ on a strip
of width $O(h)$ near the coupling boundary $\Gamma$. Estimate~\eqref{eq:thm1}
then employs this result to obtain the optimal convergence rate for the error
in the variable $\varphi$. The variational crime associated with  approximating the 
input data is assessed in Section~\ref{sec:randnaehe}.
Up to logarithmic terms, Theorem~\ref{thm:variational-crimes}
transfers the results of Theorem~\ref{thm:approximation-of-varphi} also to
this setting.
Section~\ref{sec:numerics} illustrates numerically the convergence
results of Theorem~\ref{thm:approximation-of-varphi} for several geometries. 

\section{Preliminaries and model problem.}
\label{sec:prelim}

\subsection{Notation and spaces.}
\label{sec:notation}

Let $\Omega \subset \BbbR^d$, $d \in \{2,3\}$, be a bounded Lipschitz domain
with boundary $\Gamma:=\partial\Omega$. We assume that the boundary is
polygonal/polyhedral domain with $N_\Gamma$ edges/faces $\Gamma_i$, $i=1,\ldots,N_\Gamma$. 
\footnote{The restriction to polygonal/polyhedral domains instead of smooth or piecewise smooth domains 
is not essential and due to our desire to use standard polynomial approximation results.}

For $s \in \BbbR$, we employ standard notation for the Sobolev spaces $H^s(\Omega)$ and $H^s(\Gamma_i)$, 
$i \in \{1,\ldots,N_\Gamma\}$, see, e.g., \cite{tartar07}. 
For $s > 0$, $s \not\in \BbbN_0$, we define the 
Besov spaces $B^s_{2,q}(\Omega)$ 
for  $q \in [1,\infty]$ by 
interpolation (the ``real'' method, also known as ``$K$-method'', \cite{tartar07,triebel95}): 
$$
B^s_{2,q}(\Omega) = (H^\sigma(\Omega),H^{\sigma+1}(\Omega))_{\theta,q}, 
\qquad \sigma = \lfloor s \rfloor, \quad 
\theta = s - \sigma. 
$$
We recall that for Banach spaces $X_1 \subset X_0$, the interpolation space $X_{\theta,q}= (X_0,X_1)_{\theta,q}$ 
with $\theta \in (0,1)$ and $q \in [1,\infty]$ is defined by the norm 
$\|\cdot\|_{X_{\theta,q}}$ with 
\begin{align}\label{eq:intnorm}
\|u\|_{X_{\theta,q}}:= \Big(\int_{t=0}^\infty \left| t^{-\theta} K(t,u)\right|^q \frac{dt}{t}\Big)^{1/q}, 
\quad \mbox{ for $q \in [1,\infty)$, }
\quad\mbox{ and }\quad
\|u\|_{X_{\theta,\infty}}:= \sup_{t > 0} t^{-\theta} K(t,u), 
\end{align}
where $K(t,u)= \inf_{v \in X_1} \|u - v\|_{X_0} + t \|v\|_{X_1}$.
We recall the interpolation estimate
\begin{align}\label{eq:intest}
 \|x\|_{X_{\theta,q}}
 \lesssim \|x\|_{X_0}^{1-\theta}\|x\|_{X_1}^\theta
 \quad\forall x\in X_1.
\end{align}%
For $\Gamma$, we also employ standard notation for the 
Sobolev spaces $H^s(\Gamma)$, $s \in [-1,1]$. 
For $s \ge 0$, we define the 
spaces $H^s_{pw}(\Gamma)\subset L^2(\Gamma)$ as broken spaces, i.e., 
we identify them with the product $\prod_{i=1}^{N_\Gamma} H^s(\Gamma_i)$. 
We introduce the nonstandard space 
$$
B^0_{2,\infty}(\Gamma) = (H^{-\varepsilon}(\Gamma),H^{\varepsilon}(\Gamma))_{1/2,\infty}, 
\qquad  \varepsilon \in (0,1] \mbox{ arbitrary.}
$$
(The precise choice of $\varepsilon$ is immaterial due to the 
reiteration theorem~\cite[Thm.~26.3]{tartar07}.)
Important roles in our analysis are played by the
distance function $\delta_\Gamma$, the regularized distance function 
$\widetilde \delta_\Gamma$,  and the strips $S_h$ near $\Gamma$ given by
\begin{eqnarray}
\label{eq:strip-delta} 
\delta_\Gamma(x) &:=& \operatorname*{dist}(x,\Gamma), \qquad \qquad 
\widetilde\delta_\Gamma(x):= h + \delta_\Gamma(x), \\
\label{eq:strip} 
S_h& :=& \{x \in \Omega\,|\, \delta_\Gamma(x) < h\}, 
\qquad h > 0. 
\end{eqnarray}
Naturally, properties of the trace operator $\gamma:H^1(\Omega) \rightarrow H^{1/2}(\Gamma) $ 
feature prominently in coupling procedures. We recall that $\gamma$ can be extended to $H^{1/2+\varepsilon}(\Omega)$ for all $\varepsilon > 0$ but not 
to $H^{1/2}(\Omega)$. It is, however, well-defined on the slightly smaller space 
$B^{1/2}_{2,1}(\Omega)\subset H^{1/2}(\Omega)$  
(see, e.g., \cite[Thm.~{2.9.3}]{triebel95}, \cite[Sec.~{32}]{tartar07}). 
We close this section with an embedding result that will be important to exploit additional
regularity of the solution and to make use of the smallness of the support of the right-hand side
of certain dual problems. 

\begin{lemma}
\label{lemma:weighted-embedding}
Let $\Omega \subset \BbbR^d$ be a bounded Lipschitz domain. 
Recall the regularized distance function
$\widetilde \delta_\Gamma:= \delta_\Gamma + h$ from~\eqref{eq:strip-delta}.
Then, 
\begin{eqnarray}
\label{eq:lemma:weighted-embedding-1}
\|\widetilde \delta_\Gamma^{-1/2+\varepsilon} z\|_{L^2(\Omega)} &\leq& 
\|\delta_\Gamma^{-1/2+\varepsilon} z\|_{L^2(\Omega)} \leq 
C \|z\|_{H^{1/2-\varepsilon}(\Omega)}
 \quad \forall\, 0<\varepsilon \le1/2 
\qquad \forall z \in H^{1/2-\varepsilon}(\Omega), \\
\label{eq:lemma:weighted-embedding-2}
\|\widetilde \delta_\Gamma^{-1/2} z\|_{L^2(\Omega)} &\leq& 
C |\ln h|^{1/2} \|z\|_{B^{1/2}_{2,1}(\Omega)} 
\qquad \forall z \in B^{1/2}_{2,1}(\Omega), \\
\label{eq:lemma:weighted-embedding-4}
\|\widetilde \delta_\Gamma^{-1/2-\varepsilon} z\|_{L^2(\Omega)} &\leq& 
C h^{-\varepsilon}\|z\|_{B^{1/2}_{2,1}(\Omega)}
\quad \forall \varepsilon > 0
\qquad \forall z \in B^{1/2}_{2,1}(\Omega), 
\end{eqnarray}
where $C>0$ depends only on $\Omega$ and $\varepsilon$.
\end{lemma}

\begin{proof}
The estimate involving $\delta_\Gamma$ in \eqref{eq:lemma:weighted-embedding-1}
can be found, for example, in \cite[Thm.~{1.4.4.3}]{grisvard85a}. The estimates
\eqref{eq:lemma:weighted-embedding-2}, \eqref{eq:lemma:weighted-embedding-4} follow from
1D Sobolev embedding theorems and locally flattening the boundary $\Gamma$ in the same
way as it is done in the proof of \cite[Lemma~{2.1}]{li-melenk-wohlmuth-zou10}.
\end{proof}
 
\subsection{Coupling model problem.}
\label{sec:modprob}

We denote by $n$ the normal vector on $\Gamma$ pointing into $\Omega^\extr:=\BbbR^d \setminus \overline\Omega$. 
For $d = 2$, we assume the scaling condition $\diam{(\Omega)} < 1$ so that the single layer operator 
(defined in~\eqref{eq:boundary-integral-operators} below)
is a bijection and, in fact, $H^{-1/2}(\Gamma)$-elliptic. We emphasize that $\Omega$ is connected by definition. 

Let $\opA \in L^\infty(\Omega; \BbbR^{d \times d})$ be pointwise symmetric positive definite
and satisfy $\opA \ge \alpha_0 > 0$ for some $\alpha_0 > 0$. 
We will require 
$\opA \in C^{0,1}(\overline{\Omega})$ for lowest order discretizations (the case $k=1$ below) and 
$\opA \in C^{1,1}(\overline{\Omega})$ for higher order discretizations ($k > 1$ in the following). We mention
in passing that the shift theorems Assumptions~\ref{assumption:shift-theorem} 
and \ref{assumption:shift-theorem-bidual-problem} also implicitly contain certain regularity requirements on $\opA$.

For given data $(f,u_0,\phi_0) \in L^2(\Omega) \times H^{1/2}(\Gamma) \times H^{-1/2}(\Gamma)$,
we consider the linear interface problem
\begin{subequations}\label{eq:strongform}
\begin{align}
    -\nabla \cdot (\opA\nabla u) &= f &&\quad\text{in }\Omega,
    \label{eq:strongform:interior}\\
    -\Delta u^\extr &= 0 &&\quad\text{in }\Omega^{\rm ext},
    \label{eq:strongform:exterior}\\
    u- u^\extr &= u_0 &&\quad\text{on }\Gamma,
    \label{eq:strongform:trace}\\
    (\opA\nabla u -\nabla u^{\extr})\cdot n &= \phi_0 &&\quad\text{on }\Gamma,
    \label{eq:strongform:normal}\\
    u^\extr &= O(|{x}|^{-1}) &&\quad\text{as } |x|\rightarrow \infty.
    \label{eq:strongform:radiation}
\end{align}
\end{subequations}
As usual, these equations are understood in the weak sense, i.e., we look for a
solution $(u,u^\extr)\in H^1(\Omega)\times H_{\rm loc}^1(\Omega^\extr)$, where 
$H_{\rm loc}^1(\Omega^\extr) := \{v \,\vert\, v\in H^1(K)\, ,
K\subseteq {\Omega^\extr \cup \Gamma } \,\text{ compact} \}$, 
and \eqref{eq:strongform:trace} and \eqref{eq:strongform:normal}
are understood in $H^{1/2}(\Gamma)$ and $H^{-1/2}(\Gamma)$, respectively.
It is well-known (see also Lemma~\ref{lemma:solvability} and Section~\ref{section:weakform}
below) that problem \eqref{eq:strongform} admits a unique solution in
3D. In 2D, the given data have to fulfill the compatibility
condition
\begin{align}\label{eq:compatibility2d}
 \langle f,1\rangle_\Omega + \langle\phi_0,1\rangle_\Gamma = 0
\end{align}
to ensure the behavior \eqref{eq:strongform:radiation}  of the solution at infinity.
Alternatively, one may relax the radiation condition~\eqref{eq:strongform:radiation}
to $u^{\extr} = O(\log|x|)$ as $|x|\to\infty$, see Lemma~\ref{lemma:weak-form-equals-strong-form} below.

\subsection{Operators.}
\label{sec:operators}

Let $G$ be the Green's function for the Laplacian, i.e., 
\begin{equation}
\label{eq:green}
G(x,y) = 
- \frac{1}{2\pi} \log |x - y| \quad \mbox{ if $d = 2$,} 
\qquad \qquad  
G(x,y) = 
\frac{1}{4\pi} \frac{1}{|x - y|} \quad  \mbox{ if $d = 3$.} 
\end{equation}
With 
$\partial_{n(y)}^{\rm int}$ being the interior normal derivative at $y\in\Gamma$,
we define the single layer and double layer potentials $\slp$ and $\dlp$
by 
\begin{align*}
(\slp\phi)(x)
&:=\int_\Gamma G(x,y)\phi(y)\,dS(y),
& (\dlp u)(x)
&:=\int_\Gamma \partial_{n(y)}^\intr G(x,y)u(y)\,dS(y)
 \quad\text{for }x\in\BbbR^d\setminus\Gamma.
\end{align*}
Recall that $\gamma = \gamma^\intr:H^1(\Omega)\to H^{1/2}(\Gamma)$ denotes
the (interior) trace operator. Similarly, $\gamma^\extr:H^1_{\rm loc}(\Omega^\extr)\to H^{1/2}(\Gamma)$ 
denotes the exterior trace operator.
The \emph{single layer operator}
$\slo:H^{-1/2}(\Gamma)\rightarrow H^{1/2}(\Gamma)$,  
the \emph{double layer operator} 
$\dlo :H^{1/2}(\Gamma)\rightarrow H^{1/2}(\Gamma)$, 
the \emph{adjoint double layer operator}
$\adlo:H^{-1/2}(\Gamma)\rightarrow H^{-1/2}(\Gamma)$, 
and the \emph{hypersingular operator}
$\hyp:H^{1/2}(\Gamma)\rightarrow H^{-1/2}(\Gamma)$
are defined by
\begin{align}
\label{eq:boundary-integral-operators}
\begin{aligned}
& \slo \phi := \gamma^\intr(\slp \phi) = \gamma^\extr(\slp\phi),
&\quad&
\hyp u := -\partial_n^\intr (\dlp u) = -\partial_n^\extr (\dlp u),\\
& \adlo\phi := \partial_n^\intr (\slp \phi) - 1/2\phi
= \partial_n^\extr (\slp \phi) + 1/2\phi, 
&\quad&
\dlo u := \gamma^\intr(\dlp u) + 1/2u
= \gamma^\extr(\dlp u) - 1/2u.
\end{aligned}
\end{align}
Here and throughout, we define for sufficiently smooth $v$
the exterior normal derivative
$$
\partial_n^\extr v = \gamma^\extr (\nabla v) \cdot n. 
$$

\begin{lemma}
\label{lemma:calderon}
Let $u \in H^{1/2}(\Gamma)$, $\varphi \in H^{-1/2}(\Gamma)$. Define 
\begin{equation}
\label{eq:uextr}
u^\extr:= \dlp u - \slp \varphi \qquad \text{ in }\Omega^\extr. 
\end{equation}
Then, the condition $\slo \varphi + (1/2-\dlo) u = 0$ implies 
the following assertions (\ref{item:lemma:calderon-i})--(\ref{item:lemma:calderon-ii}).
\begin{enumerate}[(i)]
\item 
\label{item:lemma:calderon-i}
$\gamma^\extr u^\extr = u$ and 
$\partial_n^\extr u^\extr = \varphi$. 
\item  
\label{item:lemma:calderon-ii}
$u^\extr$ satisfies the exterior Calder{\'o}n system: 
\begin{align}
\label{extr Calderon I}
\quad\gamma^\extr u^\extr 
& = (1/2 + \dlo)(\gamma^\extr u^\extr) 
- \slo(\partial_n^\extr u^\extr), \\
\label{extr Calderon II}
\quad\partial_n^\extr u^\extr 
& = -\hyp(\gamma^\extr u^\extr) 
+ (1/2 - \adlo)(\partial_n^\extr u^\extr). 
\end{align}
\end{enumerate}
\end{lemma}

\begin{proof}
Taking the exterior trace in~\eqref{eq:uextr}, we obtain
$$
\gamma^\extr u^\extr = (1/2 + \dlo )u  - \slo \varphi  = u + (-1/2 + \dlo )u - \slo \varphi = u. 
$$
A calculation (which is non-trivial for the 2D case) 
shows that $u^\extr$ satisfies 
the following, second representation formula (see, e.g., the proof of \cite[Lemma~{2.3}]{aurada-melenk-praetorius13a}
for details): 
$$
u^\extr = \dlp \gamma^\extr u^\extr - \slp \partial_n^\extr u^\extr. 
$$
The exterior Calder{\'o}n system then follows from taking the trace and the trace of the (exterior)
normal derivative. Finally, the bijectivity of $\slo$ implies
$\partial_n^\extr u^\extr = \varphi$.  
\end{proof}

\subsection{Bilinear forms.} 
\label{sec:bilinear-forms}

To state the FEM-BEM coupling~\eqref{eq:primal-block-system}
for the model problem~\eqref{eq:strongform},
we define the following four bilinear forms:
\begin{align*}
a(u,v) &:= \langle \opA \nabla u ,\nabla v \rangle_\Omega, 
& 
\widetilde a(u,v) &:= a(u,v) + \langle \hyp u,v\rangle_\Gamma, \\
c(\varphi,\psi) &:= \langle \slo \varphi, \psi\rangle_\Gamma, 
& 
b(u,\psi) &:= \langle (1/2 - \dlo) u,\psi\rangle_\Gamma.  
\end{align*}
The bilinear form $c(\cdot,\cdot)$ induces a norm that is equivalent to the $H^{-1/2}(\Gamma)$-norm: 
$\|\psi\|_V^2:= c(\psi,\psi) \sim \|\psi\|^2_{H^{-1/2}(\Gamma)}$. 
For bounded linear functionals $L_1:H^1(\Omega)\to\mathbb R$ and $L_2:H^{-1/2}(\Gamma)\to\mathbb R$, 
we consider the block system 
\begin{subequations}
\label{eq:primal-block-system}
\begin{eqnarray}
\widetilde a(u,v) - b(v,\varphi) &=& L_1(v)  \qquad \forall v \in H^1(\Omega),\\
b(u,\psi) + c(\varphi,\psi) &=& L_2(\psi)  \qquad \forall \psi \in H^{-1/2}(\Gamma).
\end{eqnarray}
We set 
\begin{align}
X := H^1(\Omega) \times H^{-1/2}(\Gamma). 
\end{align}
\end{subequations}
The Galerkin formulation is obtained in the usual fashion: For a conforming 
subspace
\begin{subequations}
\label{eq:primal-block-system:h}
\begin{align}
X_h := V_h \times M_h \subset X,
\end{align}%
we define $(u_h,\varphi_h) \in X_h$ by requiring that 
\eqref{eq:primal-block-system} be satisfied with the spaces 
$H^1(\Omega)$ and $H^{-1/2}(\Gamma)$ replaced with $V_h$ and $M_h$, i.e.,
\begin{eqnarray}
\widetilde a(u_h,v) - b(v,\varphi_h) &=& L_1(v)  \qquad \forall v \in V_h,\\
b(u_h,\psi) + c(\varphi_h,\psi) &=& L_2(\psi)  \qquad \forall \psi \in M_h.
\end{eqnarray}
\end{subequations}%
We note that both the system \eqref{eq:primal-block-system} and its discrete counterpart have 
unique solutions for any $(L_1,L_2) \in X^\prime$ 
by coercivity properties of the  left-hand side of \eqref{eq:primal-block-system}, which are collected 
in the following lemma. 

\begin{lemma}
\label{lemma:solvability}
The bilinear form 
$\displaystyle 
A((u,\varphi); (v,\psi)):= \widetilde a(u,v) - b(v,\phi) + b(u,\psi) + c(\varphi,\psi)
$
satisfies: 
\begin{enumerate}[(i)]
\item
\label{lemma:solvability:item:i}
$A(\cdot,\cdot)$ is semi-definite: 
$\displaystyle 
A((u,\varphi);(u,\varphi)) \ge C \Big[ \|\nabla u\|^2_{L^2(\Omega)} + \|\varphi\|^2_V\Big]
$, where $C>0$ depends only on $\Gamma$ and the coercivity constant $\alpha_0$ of $\mathfrak A$.
\item 
\label{lemma:solvability:item:ii}
$A(\cdot,\cdot)$ satisfies a G{\aa}rding inequality. 
\item 
\label{lemma:solvability:item:iii}
The operator ${\mathbf A}:X \rightarrow X^\prime$ 
induced by $A(\cdot,\cdot)$ is injective. 
\item
\label{lemma:solvability:item:iv}
The operator ${\mathbf A}:X \rightarrow X^\prime$ induced by $A(\cdot,\cdot)$
satisfies an inf-sup condition on $X$.
\item 
\label{lemma:solvability:item:vi}
Fix $\xi \in H^{-1/2}(\Gamma)$ with $\langle \xi,1\rangle_\Gamma \ne 0$
(in particular, $\xi \equiv 1$ is admissible). 
For any conforming discretization $X_h\subset X$ that contains 
$(0,\xi)$, the discrete inf-sup condition is (uniformly) satisfied, i.e., 
the discrete inf-sup constant depends solely on $\mathfrak A$ and the choice of 
$\xi$,
but is independent of $h$.
In particular, \eqref{eq:primal-block-system} as well as~\eqref{eq:primal-block-system:h} 
admit unique solutions $(u,\phi)\in X$ and $(u_h,\phi_h)\in X_h$. Moreover,
there is a constant $C > 0$ depending only on $\mathfrak A$ and $\xi$, such that 
the following quasi-optimality result holds for the solutions $(u,\phi)$ of 
\eqref{eq:primal-block-system} and 
its discrete approximation 
$(u_h,\phi_h) \in X_h$:  
\begin{equation}
\label{eq:quasi-optimality}
\|u - u_h\|_{H^1(\Omega)} + \|\varphi - \varphi_h\|_{V} \leq C \inf_{(v_h,\psi_h) \in X_h} 
\left[\|u - v_h\|_{H^1(\Omega)} + \|\varphi - \psi_h\|_{V}\right]. 
\end{equation}
\end{enumerate}
\end{lemma}

\begin{proof}
Obviously, \eqref{lemma:solvability:item:i} is satisfied. 
The observation~\eqref{lemma:solvability:item:ii} that ${\mathbf A}$ satisfies a G{\aa}rding inequality
was first made in \cite{carstensen-stephan95a}. (It is also found in the seminal works~\cite{costabel88a,han90}, 
where an additional Dirichlet boundary is assumed.) Together with the injectivity
statement of \eqref{lemma:solvability:item:iii}, the inf-sup condition~\eqref{lemma:solvability:item:iv}
for $A(\cdot,\cdot)$ holds.
Item~\eqref{lemma:solvability:item:vi} is shown in 
\cite{aurada-feischl-fuehrer-karkulik-melenk-praetorius13}. 
The quasi-optimality assertion \eqref{eq:quasi-optimality} is a consequence of the 
(uniform) discrete inf-sup condition. 

Finally, let us discuss the injectivity~\eqref{lemma:solvability:item:iii} of ${\mathbf A}$. Starting from 
$$
A((u,\varphi); (v,\psi)) = 0 \qquad \forall (v,\psi) \in X,
$$
we get from \eqref{lemma:solvability:item:i} that $u$ is constant 
($\Omega$ is connected!) and $\varphi = 0$. This implies 
in view of $b(u,\psi) = 0$ for all $\psi$ and the well-known fact $(-1/2 + K) 1 = -1$ that 
$$
0 = (1/2 - \dlo) u  \stackrel{u \ const}{=} u.
$$
\end{proof}


\subsection{Weak formulation and Galerkin approximation.}
\label{section:weakform}

For the original problem \eqref{eq:strongform}, it is convenient to 
introduce the linear forms 
\begin{subequations}
\label{eq:rhs-weak}
\begin{eqnarray}
\label{eq:rhs-weak-1}
L_1(v) &:=& \langle f,v\rangle_\Omega + \langle \phi_0 + \hyp u_0,v\rangle_\Gamma,\\
\label{eq:rhs-weak-2}
L_2(\psi) &:=& \langle \psi,(1/2-\dlo) u_0\rangle_\Gamma. 
\end{eqnarray}
\end{subequations}
Our weak formulation of \eqref{eq:strongform} is: 
Find $(u,\varphi) \in X$ such that 
\begin{subequations}
\label{eq:weak-form}
\begin{eqnarray}
\widetilde a(u,v) - b(v,\varphi) &=& L_1(v) \qquad \forall v \in H^1(\Omega), \\
b(u,\psi) + c(\varphi,\psi) &=& L_2(\psi) \qquad \forall \psi \in H^{1/2}(\Gamma) . 
\end{eqnarray}
\end{subequations}
The Galerkin approximation is correspondingly given by 
\begin{subequations}
\label{eq:weak-form-FEM}
\begin{eqnarray}
\widetilde a(u,v) - b(v,\varphi) &=& L_1(v) \qquad \forall v \in V_h, \\
b(u,\psi) + c(\varphi,\psi) &=& L_2(\psi) \qquad \forall \psi \in M_h.
\end{eqnarray}
\end{subequations}
By Lemma~\ref{lemma:solvability}, we have unique solvability of \eqref{eq:weak-form}. The following lemma clarifies in what sense it solves
\eqref{eq:strongform}: 

\begin{lemma}
\label{lemma:weak-form-equals-strong-form}
Let $(u,\varphi) \in X$ solve \eqref{eq:rhs-weak}--\eqref{eq:weak-form}. Define 
the solution $u^\extr$ in $\Omega^\extr$ by 
\begin{equation}
\label{eq:uextr-with-jumps}
u^\extr := \dlp (u - u_0) - \slp \varphi. 
\end{equation}
Then 
the conditions 
\eqref{eq:strongform:interior}--\eqref{eq:strongform:normal}
are satisfied (in the appropriate senses). Concerning the 
radiation condition at $\infty$, we have: 
\begin{itemize}
\item For $d =  3$, the radiation condition \eqref{eq:strongform:radiation}
is satisfied.
\item 
For $d = 2$, the radiation condition \eqref{eq:strongform:radiation}
is satisfied if the data $f$, $\phi_0$ satisfy the compatibility condition 
\begin{align}
\label{eq:2d-compatibility}
\langle f,1\rangle_\Omega + \langle \phi_0,1\rangle_\Gamma = 0. 
\end{align}%
\item 
If $d=2$ and the compatibility condition \eqref{eq:2d-compatibility} is not fulfilled, then 
the solution $u^\extr$ satisfies 
\begin{equation}
\label{eq:uext-2D-asymptotics}
u^\extr(x) = a 
 \log |x| + O(1/|x|)  \quad \mbox{ as } \quad |x| \rightarrow \infty,
\qquad 
 \mbox{ where $a:= \langle f,1\rangle_\Omega + \langle \phi_0,1\rangle_\Gamma$}. 
\end{equation}
\end{itemize}
\end{lemma}

\begin{proof}
We first show 
\eqref{eq:strongform:interior}--\eqref{eq:strongform:normal}. 
Integration by parts and varying the test functions yield
\begin{eqnarray}
\label{eq:lemma:weak-form-equals-strong-form-1}
-\nabla \cdot (\opA \nabla u) &=& f \quad \mbox{ in $\left(H^1_0(\Omega)\right)^\prime$},\\
\label{eq:lemma:weak-form-equals-strong-form-2}
(\opA \nabla u) \cdot n + \hyp u - (1/2 - \adlo)\varphi &=& \phi_0 + \hyp u_0 
\qquad \mbox{ in $H^{-1/2}(\Gamma)$},\\
\label{eq:lemma:weak-form-equals-strong-form-3}
(1/2 - \dlo) u + \slo \varphi &=& (1/2-\dlo) u_0 
\qquad \mbox{ in $H^{1/2}(\Gamma)$}.
\end{eqnarray}
{}From 
\eqref{eq:lemma:weak-form-equals-strong-form-3} and Lemma~\ref{lemma:calderon},
we obtain that the function $u^\extr$ defined 
in \eqref{eq:uextr-with-jumps} has the following traces 
on $\Gamma$: 
\begin{align*}
\gamma^\extr u^\extr = u - u_0, 
\qquad
\partial_n^\extr u^\extr = \varphi. 
\end{align*}
In particular, this proves~\eqref{eq:strongform:trace}.
Furthermore, Lemma~\ref{lemma:calderon} (\ref{item:lemma:calderon-ii}) shows
\begin{eqnarray*}
\varphi &=& - \hyp \gamma^\extr u^\extr + (1/2 - \adlo) \varphi.
\end{eqnarray*}
Upon insertion into 
\eqref{eq:lemma:weak-form-equals-strong-form-2}, this gives
\begin{eqnarray*}
(\opA \nabla u) \cdot n &=&  
- \hyp u + (1/2 - \adlo)\varphi + \phi_0 + \hyp u_0 
= - \hyp \gamma^\extr u^\extr + \varphi + \hyp \gamma^\extr u^\extr
+ \phi_0
= \partial_n^\extr u^\extr + \phi_0,
\end{eqnarray*}
which is \eqref{eq:strongform:normal}. 

In 3D, the radiation condition \eqref{eq:strongform:radiation}
follows from the decay properties at $\infty$ of the potentials  $\slp$ and $\dlp$. 
In 2D, this decay property at $\infty$ follows from the properties 
of $\slp$ and $\dlp$ {\em if} $\langle \varphi,1\rangle_\Gamma = 0$. 
The compatibility condition \eqref{eq:2d-compatibility} implies this 
with the test function $v \equiv 1$. 
Finally, for $d=2$ 
and the case when \eqref{eq:2d-compatibility} is not satisfied, 
then by \eqref{eq:uextr-with-jumps}
the leading order behavior of $u^\extr$ (as $|x| \rightarrow \infty$) 
is clearly $-\langle 1,\varphi\rangle_\Gamma \log |x|$; choosing 
$v \equiv 1$ as a test function yields \eqref{eq:uext-2D-asymptotics}.
\end{proof}

\subsection{The dual problem.}
\label{sec:dual}

Our FEM analysis will rely on various dual problems. 
The first dual problem that we consider is: 
Find $(w,\lambda) \in X$ such that 
\begin{subequations}
\label{eq:dual-problem}
\begin{eqnarray}
\widetilde a(v,w) - b(v,\lambda) &=& f(v) \qquad \forall v \in H^1(\Omega), \\
b(w,\psi) + c(\psi,\lambda)&=& 0 \qquad \forall \psi \in H^{-1/2}(\Gamma) .
\end{eqnarray}
\end{subequations} 
Due to symmetry of $\widetilde a(\cdot,\cdot)$ and $c(\cdot,\cdot)$, 
Lemma~\ref{lemma:solvability} applies and proves existence and uniqueness of
$(w,\lambda)\in X$. 
We denote the corresponding solution operator by 
\begin{align}\label{eq:solve:27}
&T^\dual : (H^1(\Omega))^\prime  \rightarrow X,
\qquad 
f \mapsto (T^w f,T^\lambda f) := (w,\lambda).
\end{align}
If the right-hand side $f$ has the form $f(v) = \langle f,v\rangle_\Omega$
for an $f \in L^2(\Omega)$, then the above 
Lemma~\ref{lemma:weak-form-equals-strong-form}
shows that 
$(w,\lambda)$ satisfies the transmission problem 
\eqref{eq:strongform} with $u_0 = 0$, $\phi_0 = 0$ and, in 2D, 
the radiation condition 
$$
u^\extr = a \log |x| + O(1/|x|), \qquad |x| \rightarrow \infty 
\qquad\mbox{ for some $a \in \BbbR$.}
$$
Hence, \eqref{eq:dual-problem} it is a classical transmission problem for which we will 
make the following assumption: 

\begin{assumption}
\label{assumption:shift-theorem}
There exist $s_0 \in (1/2,1]$  and $C>0$ such that the mapping 
$f \mapsto T^\dual f = (T^w f,T^\lambda f)$ from~\eqref{eq:solve:27} satisfies 
$$
\|T^w f\|_{H^{1+s_0}(\Omega)} + \|T^\lambda f\|_{H^{-1/2+s_0}_{pw}(\Gamma)} 
\leq C \|f\|_{(H^{1-s_0}(\Omega))^\prime}. 
$$
\end{assumption}

\begin{remark}
{\normalfont 
Assumption~\ref{assumption:shift-theorem} is satisfied in the following
simple cases: 
\begin{enumerate}
\item 
The coefficient matrix $\opA$ is smooth and $\Gamma$ is sufficiently
smooth. Then, by  classical regularity theory, $s_0 = 1$ is possible. 
\item 
In 2D, $\Omega$ is a polygon and the matrix $\opA$ 
has the form $\opA(x) = a(x) \operatorname*{Id}$ for a scalar-valued 
function $a$ that is sufficiently smooth. See, e.g., 
\cite{costabel-stephan85a}, 
\cite[Appendix]{costabel-dauge-nicaise99}, \cite{mercier03}. 
\item 
The discussion in \cite[Rem.~{5.1}]{elschner-kaiser-rehberg-schmidt07} shows
that the shift theorem of Assumption~\ref{assumption:shift-theorem} 
is in general false for piecewise smooth, pointwise symmetric positive definite matrices $\opA$. \eremk
\end{enumerate}}%
\end{remark}

Recall the definition of $S_h$ from~\eqref{eq:strip}.
We have the following regularity result. 

\begin{lemma}
\label{lemma:B32-regularity}
Let Assumption~\ref{assumption:shift-theorem} be valid. Then 
the operator $T^\dual: (H^1(\Omega))^\prime \rightarrow X$ 
from~\eqref{eq:solve:27} satisfies 
\begin{eqnarray}
\label{eq:lemma:B32-regularity-1}
\|T^w f\|_{B^{3/2}_{2,\infty}(\Omega)}  + 
\|T^\lambda f\|_{B^{0}_{2,\infty}(\Gamma)}  
\leq C \|f\|_{(B^{1/2}_{2,1}(\Omega))^\prime}. 
\end{eqnarray}
In particular, if $f \in L^2(\Omega)$ with $\supp f \subset \overline{S_{h}}$, 
then 
\begin{eqnarray}
\label{eq:lemma:B32-regularity-2}
\|T^w f\|_{B^{3/2}_{2,\infty}(\Omega)}  + 
\|T^\lambda f\|_{B^{0}_{2,\infty}(\Gamma)}  & \leq & C h^{1/2} \|f\|_{L^2(\Omega)},  \\
\label{eq:lemma:B32-regularity-3}
\|T^w f\|_{H^{3/2+\varepsilon}(\Omega)}  + 
\|T^\lambda f\|_{H^{\varepsilon}_{pw}(\Gamma)}  &\leq & C h^{1/2-\varepsilon} \|f\|_{L^2(\Omega)}
\qquad \forall\, 0 < \varepsilon \leq s_0-1/2. 
\end{eqnarray}
The constant $C>0$ in~\eqref{eq:lemma:B32-regularity-1}--\eqref{eq:lemma:B32-regularity-2} depends only on $\Omega$ and Assumption~\ref{assumption:shift-theorem}, while that of~\eqref{eq:lemma:B32-regularity-3}
depends additionally on $\varepsilon$.
\end{lemma}

\begin{proof}
We follow the arguments of \cite[Lemma~{5.2}]{melenk-wohlmuth12}. 
The starting point for the proof of \eqref{eq:lemma:B32-regularity-1} is that interpolation and 
Assumption~\ref{assumption:shift-theorem} yield with $\theta = 1/(2s_0) \in (0,1)$ 
and hence $s_0\theta=1/2$ well-posedness and stability of 
$$
T^w: ((H^1(\Omega))^\prime,(H^{1-s_0}(\Omega))^\prime)_{\theta,\infty} 
\rightarrow 
(H^1(\Omega),H^{1+s_0}(\Omega))_{\theta,\infty}  = 
B^{1+s_0\theta}_{2,\infty}(\Omega) 
= B^{3/2}_{2,\infty}(\Omega).
$$
The arguments for  $T^\lambda$ proceed along the same lines,
but rely on the mapping properties $T^\lambda: (H^1(\Omega))^\prime  \rightarrow H^{-1/2}(\Gamma) $
as well as 
$T^\lambda: (H^{1-s_0}(\Omega))^\prime  \rightarrow H^{-1/2+s_0}_{pw}(\Gamma)$. 
Assuming, as we may, that $s_0 < 1$, we have $0<-1/2 + s_0 < 1/2$. Hence, 
$H^{-1/2+s_0}(\Gamma)$ is isomorphic to the product space $\prod_{i=1}^{N_\Gamma} H^{-1/2+s_0}(\Gamma_i)\equiv H^{-1/2+s_0}_{pw}(\Gamma)$. 
Thus, interpolation proves well-posedness and stability of
$$
T^\lambda:((H^1(\Omega))^\prime,(H^{1-s_0}(\Omega))^\prime)_{\theta,\infty}\rightarrow 
(H^{-1/2}(\Gamma),H^{-1/2+s_0}(\Gamma))_{\theta,\infty}
= B^{-1/2+s_0\theta}_{2,\infty}(\Gamma)
= B^0_{2,\infty}(\Gamma).
$$
As in \cite[Lemma~{5.2}]{melenk-wohlmuth12}
(cf.~\cite[Thm.~{1.11.2}]{triebel95} or \cite[Lemma~{41.3}]{tartar07}),
we recognize that
$$
((H^1(\Omega))^\prime,(H^{1-s_0}(\Omega))^\prime)_{\theta,\infty}  = 
(B^{1-s_0\theta}_{2,1}(\Omega))^\prime
= (B^{1/2}_{2,1}(\Omega))'.
$$ 
The combination of the last three observations proves~\eqref{eq:lemma:B32-regularity-1}.
The proof of \eqref{eq:lemma:B32-regularity-2} follows by the same argument
as in \cite[Lemma~{5.2}]{melenk-wohlmuth12}. 
To prove~\eqref{eq:lemma:B32-regularity-3}, we first 
note that the case $\varepsilon=s_0-1/2$ coincides with Assumption~\ref{assumption:shift-theorem}.
For $0<\varepsilon<s_0-1/2$, we argue
as for~\eqref{eq:lemma:B32-regularity-1}. 
Interpolation with $0<\theta<1$ and $s_0\theta=1/2+\varepsilon$ yields
$$
\|T^w f\|_{H^{3/2+\varepsilon}(\Omega)} + 
\|T^\lambda f\|_{H^{\varepsilon}_{pw}(\Gamma)} 
\lesssim \|f\|_{(H^{1/2-\varepsilon}(\Omega))^\prime},
$$
where we again used $H^\varepsilon(\Gamma)\equiv H^\varepsilon_{\rm pw}(\Gamma)$ as $0<\varepsilon<s_0-1/2\leq 1/2$.
Next, we apply estimate \eqref{eq:lemma:weighted-embedding-1} from
Lemma~\ref{lemma:weighted-embedding} to see
\begin{align}
\nonumber 
\|f\|_{(H^{1/2-\varepsilon}(\Omega))^\prime} &= 
\sup_{v \in H^{1/2-\varepsilon}(\Omega)} \frac{\langle f,v\rangle}{\|v\|_{H^{1/2-\varepsilon}(\Omega)}} 
= 
\sup_{v \in H^{1/2-\varepsilon}(\Omega)} \frac{\langle \delta_\Gamma^{1/2-\varepsilon} f, \delta_\Gamma^{-(1/2-\varepsilon)} v\rangle}{\|v\|_{H^{1/2-\varepsilon}(\Omega)}} 
\leq \|\delta_\Gamma^{1/2-\varepsilon} f\|_{L^2(\Omega)} 
\sup_{v \in H^{1/2-\varepsilon}(\Omega)} \frac{\|\delta_\Gamma^{-(1/2-\varepsilon)} v\|_{L^2(\Omega)}}{\|v\|_{H^{1/2-\varepsilon}(\Omega)}} \\
\label{eq:lemma:B32-regularity-20}
&\lesssim h^{1/2-\varepsilon} \|f\|_{L^2(\Omega)},
\end{align}
where in the last inequality we exploited the support property of $f$.
\end{proof} 
The $u$-components of the solutions of \eqref{eq:primal-block-system} and of \eqref{eq:dual-problem}
solve classical elliptic problem that feature interior regularity. We formulate this
in analogy to the corresponding result in \cite[Lemma~{5.4}]{melenk-wohlmuth12}
and \cite[Lemma~{2.7}]{melenk-wohlmuth14a}:

\begin{lemma}
\label{lemma:5.4}
Let $z$ solve 
$$
-\nabla \cdot \left(\opA \nabla z\right) = v \quad \mbox{ in $\Omega$}
$$
for some $v \in L^2(\Omega)$ with $\supp v \subset \overline{S_h}$. 
Then there are constants $c$, $\tilde c$, $c^\prime >0$ that depend solely on $\Omega$, 
such that 
for all sufficiently small $h>0$ the following assertions (\ref{item:lemma:5.4-1})--(\ref{item:lemma:5.4-4}) hold:
\begin{enumerate}[(i)]
\item 
\label{item:lemma:5.4-1}
If $z \in B^{3/2}_{2,\infty}(\Omega)$, then 
$\|\delta_\Gamma^{1/2} \nabla^2 z\|_{L^2(\Omega\setminus S_{\tilde ch})} 
\leq C_1 \sqrt{|\ln h|} \|z\|_{B^{3/2}_{2,\infty}(\Omega)}. 
$
The constant $C_1$ depends only on $\Omega$, $\alpha_0$,  and $\|\opA\|_{C^{0,1}(\overline{\Omega})}$. 
\item 
\label{item:lemma:5.4-2}
For every $\alpha > 0$, there holds 
$$
\|\delta_\Gamma^{\alpha} \nabla^3 z\|_{L^2(\Omega\setminus S_{\tilde ch})} 
\leq C_2  \|\delta_\Gamma^{\alpha-1}\nabla^2 z\|_{L^2(\Omega\setminus S_{c^\prime h})}
+ 
\widetilde C_2  \|\delta_\Gamma^{\alpha}\nabla z\|_{L^2(\Omega\setminus S_{c^\prime h})}
$$ 
The constant $C_2$ depends only on $\Omega$, $\alpha_0$, and $\|\opA\|_{C^{0,1}(\overline{\Omega})}$; 
and 
the constant $\widetilde C_2$ depends only on $\Omega$, $\alpha_0$, and $\|\opA\|_{C^{1,1}(\overline{\Omega})}$. 
\item 
\label{item:lemma:5.4-3}
If $z \in H^{3/2+\varepsilon}(\Omega)$ for some $\varepsilon \in (0,1/2)$, then 
$\|\nabla^2 z\|_{L^2(\Omega\setminus S_{\tilde c h})}
\leq C_3\, h^{-1/2+\varepsilon}
\|z\|_{H^{3/2+\varepsilon}(\Omega)} $. The constant $C_3>0$ depends only on 
$\Omega$, $\alpha_0$, $\|\opA\|_{C^{0,1}(\overline{\Omega})}$, and $\varepsilon$. 
\item 
\label{item:lemma:5.4-4}
If Assumption~\ref{assumption:shift-theorem} is valid and if 
$z = T^w v$ with $T^w$  being the first component of the solution operator $T^{\rm dual}$ from \eqref{eq:solve:27},
then 
$\|\nabla^2 z\|_{L^2(\Omega\setminus S_{ch})} 
\leq C_4 \|v\|_{L^2(\Omega)}.
$
The constant $C_4>0$ depends only on $\Omega$ and $\opA$ through the coercivity constant $\alpha_0$ of $\opA$ and 
$\|\opA\|_{C^{0,1}(\overline{\Omega})}$.
\end{enumerate}
\end{lemma}

\begin{proof} (See also \cite[Lemma~{2.9}]{melenk-wohlmuth14a}.) 

{\em Proof of \eqref{item:lemma:5.4-1}, \eqref{item:lemma:5.4-2}:}
\cite[Lemma~{5.4}]{melenk-wohlmuth12} is formulated for $-\Delta$. 
However, the essential property of the differential operator 
$-\Delta$
that is required, is just interior regularity. Hence, the result also
stands for the present, more general elliptic operator 
$-\nabla\cdot(\opA\nabla)$. The precise dependence on the coefficient $\opA$ 
is taken from \cite[Thm.~{8.10}]{gilbarg-trudinger77a}. 

{\em Proof of \eqref{item:lemma:5.4-3}:} This follows again by local 
considerations similar to those employed in the proof of 
\cite[Lemma~{5.4}]{melenk-wohlmuth12} and the crude bound 
$\delta_\Gamma \gtrsim h$ on $\Omega\setminus S_{\tilde ch}$. 

{\em Proof of \eqref{item:lemma:5.4-4}:} In view of \eqref{item:lemma:5.4-3}, 
we have to estimate $\|z\|_{H^{3/2+\varepsilon}(\Omega)}$. By the 
support properties of $v$, the bound \eqref{eq:lemma:B32-regularity-3}
yields $\|z\|_{H^{3/2+\varepsilon}(\Omega)} \leq C\, h^{1/2-\varepsilon} \|v\|_{L^2(\Omega)}$. 
Inserting this in \eqref{item:lemma:5.4-3} produces the result.
\end{proof}
%


\subsection{The bidual problem.}
\label{sec:bidual}

Similar to the procedure in \cite{melenk-wohlmuth14a}, the analysis of the 
discretization of 
the dual problem requires estimates in norms other than the standard energy-like norm. This analysis
therefore requires a second class of problems, which we 
call the ``bidual'' problem. It is given as follows: 
Find $(w,\lambda) \in X$ such that 
\begin{subequations}
\label{eq:bidual-problem}
\begin{eqnarray}
\widetilde a(w,v) - b(v,\lambda) &=& f(v) \qquad \forall v \in H^1(\Omega), \\
b(w,\psi) + c(\lambda,\psi) &=& 0 \qquad \forall \psi \in H^{-1/2}(\Gamma),
\end{eqnarray}
\end{subequations}
with solution operator $T^\bidual:f \mapsto (\widetilde w,\widetilde \lambda)$. 
In view of the symmetry of the bilinear forms 
$\widetilde a(\cdot,\cdot)$ and $c(\cdot,\cdot)$, problem \eqref{eq:bidual-problem} is, of course, essentially
the same as the dual problem \eqref{eq:dual-problem}. 
Thus, Assumption~\ref{assumption:shift-theorem} holds for \eqref{eq:bidual-problem} if 
it does for \eqref{eq:dual-problem}. Nevertheless, in order to emphasize the structure 
of the regularity requirements in our convergence theory, we formulate this shift theorem as a
separate assumption. 

\begin{assumption}
\label{assumption:shift-theorem-bidual-problem}
There exist $s_0 \in (1/2,1]$  and $C < \infty$ such that the mapping $f \mapsto T^\bidual(f) = (w,\lambda)$ 
given by \eqref{eq:bidual-problem} satisfies 
$$
\|w\|_{H^{1+s_0}(\Omega)} + \|\lambda \|_{H^{-1/2+s_0}_{pw}(\Gamma)} \leq C \|f\|_{(H^{1-s_0}(\Omega))^\prime}. 
$$
\end{assumption}

Our analysis will require an understanding of the Galerkin error for certain dual problems. 
This in turn will lead to a bidual problem with right-hand sides in weighted spaces, which we 
now analyze. 

\begin{lemma} 
\label{lemma:regularity-weighted-rhs-vorn}
Let Assumption~\ref{assumption:shift-theorem-bidual-problem} be valid. 
Recall the regularized distance function
$\widetilde \delta_\Gamma:= \delta_\Gamma + h$ from~\eqref{eq:strip-delta}.
Let $v \in L^2(\Omega)$ and $0 < \varepsilon \leq s_0-1/2$.
Then, the function $(w,\lambda) =T^\bidual(\widetilde \delta_\Gamma^{-1} v)$ satisfies
\begin{align}
\label{eq:lemma:regularity-weighted-rhs-vorn-1}
\|w\|_{B^{3/2}_{2,\infty}(\Omega)}+ 
\|\lambda\|_{B^{0}_{2,\infty}(\Gamma)}
&\leq C |\ln h|^{1/2} \|\widetilde\delta_\Gamma^{-1/2} v\|_{L^2(\Omega)}, \\ 
\label{eq:lemma:regularity-weighted-rhs-vorn-1a}
\|w\|_{H^{3/2+\varepsilon}(\Omega) }+ 
\|\lambda\|_{H^{\varepsilon}_{pw}(\Gamma) }
&\leq C h^{-\varepsilon} \|\widetilde\delta_\Gamma^{-1/2} v\|_{L^2(\Omega)}.
\end{align}
Moreover, the function
$(w^\prime ,\lambda^\prime) = T^\bidual (\widetilde \delta_\Gamma^{-1+2\varepsilon} v)$
satisfies
\begin{eqnarray}
\label{eq:lemma:regularity-weighted-rhs-vorn-2}
\|w^\prime\|_{H^{3/2+\varepsilon}(\Omega) }+ 
\|\lambda^\prime\|_{H^{\varepsilon}_{pw}(\Gamma) }
\leq C \|\widetilde\delta_\Gamma^{-1/2+\varepsilon} v\|_{L^2(\Omega)}.  
\end{eqnarray}
The constant $C>0$ in~\eqref{eq:lemma:regularity-weighted-rhs-vorn-1} depends only on $\Omega$ and Assumption~\ref{assumption:shift-theorem-bidual-problem}, while that of~\eqref{eq:lemma:regularity-weighted-rhs-vorn-1a}--\eqref{eq:lemma:regularity-weighted-rhs-vorn-2} depends additionally on $\varepsilon$.
\end{lemma}

\begin{proof}
We proceed as in \cite[Lemma~{5.2}]{melenk-wohlmuth12}. In order to 
prove \eqref{eq:lemma:regularity-weighted-rhs-vorn-1}, we employ
Assumption~\ref{assumption:shift-theorem-bidual-problem} and argue
as in Lemma~\ref{lemma:B32-regularity} to see
\begin{align*}
 \|w\|_{B^{3/2}_{2,\infty}(\Omega)}  + 
\|\lambda\|_{B^{0}_{2,\infty}(\Gamma)}  
\lesssim \|\widetilde\delta_\Gamma^{-1}v\|_{(B^{1/2}_{2,1}(\Omega))^\prime}. 
\end{align*}
Then, we compute
\begin{eqnarray*} 
\|\widetilde \delta_\Gamma^{-1} v\|_{(B^{1/2}_{2,1}(\Omega))^\prime}
&=& 
\sup_{z \in B^{1/2}_{2,1}(\Omega)} \frac{\langle \widetilde \delta_\Gamma^{-1} v,z\rangle}{\|z\|_{B^{1/2}_{2,1}(\Omega)}}
= 
\sup_{z \in B^{1/2}_{2,1}(\Omega)} \frac{\langle \widetilde \delta_\Gamma^{-1/2} v,\widetilde \delta_\Gamma^{-1/2}z\rangle}{\|z\|_{B^{1/2}_{2,1}(\Omega)}}
\lesssim \|\widetilde \delta_\Gamma^{-1/2} v\|_{L^2(\Omega)} 
\sup_{z \in B^{1/2}_{2,1}(\Omega)} \frac{\|\widetilde \delta_\Gamma^{-1/2}z\|_{L^2(\Omega)}}{\|z\|_{B^{1/2}_{2,1}(\Omega)}}. 
\end{eqnarray*}
The application of estimate \eqref{eq:lemma:weighted-embedding-2} of Lemma~\ref{lemma:weighted-embedding} concludes the argument.
For the estimate 
\eqref{eq:lemma:regularity-weighted-rhs-vorn-2}, we proceed similarly.
First, Assumption~\ref{assumption:shift-theorem-bidual-problem} and
interpolation yield
\begin{align*}
\|w^\prime\|_{H^{3/2+\varepsilon}(\Omega) }+ 
\|\lambda^\prime\|_{H^{\varepsilon}_{pw}(\Gamma) }
\lesssim 
\|\widetilde \delta_\Gamma^{-1+2\varepsilon} v\|_{(H^{1/2-\varepsilon}(\Omega))^\prime}. 
\end{align*}
Second, we compute
\begin{eqnarray*} 
\|\widetilde \delta_\Gamma^{-1+2\varepsilon} v\|_{(H^{1/2-\varepsilon}(\Omega))^\prime}
&=& 
\sup_{z \in H^{1/2-\varepsilon}(\Omega)} \frac{\langle \widetilde \delta_\Gamma^{-1+2\varepsilon} v,z\rangle}{\|z\|_{H^{1/2-\varepsilon}(\Omega)}}
= 
\sup_{z \in H^{1/2-\varepsilon}(\Omega)} \frac{\langle \widetilde \delta_\Gamma^{-1/2+\varepsilon} v,\widetilde \delta_\Gamma^{-1/2+\varepsilon}z\rangle}{\|z\|_{H^{1/2-\varepsilon}(\Omega)}} 
\lesssim \|\widetilde \delta_\Gamma^{-1/2+\varepsilon} v\|_{L^2(\Omega)} 
\sup_{z \in H^{1/2-\varepsilon}(\Omega)} 
\frac{\|\widetilde \delta_\Gamma^{-1/2+\varepsilon}z\|_{L^2(\Omega)}}{\|z\|_{H^{1/2-\varepsilon}(\Omega)}}. 
\end{eqnarray*}
An application of estimate \eqref{eq:lemma:weighted-embedding-1} of Lemma~\ref{lemma:weighted-embedding} 
leads to~\eqref{eq:lemma:regularity-weighted-rhs-vorn-2}.
Finally, we show \eqref{eq:lemma:regularity-weighted-rhs-vorn-1a}.
First, Assumption~\ref{assumption:shift-theorem-bidual-problem} and
interpolation yield
\begin{align*}
\|w\|_{H^{3/2+\varepsilon}(\Omega)}
+ \|\lambda\|_{H^\varepsilon_{pw}(\Gamma)} 
\lesssim \|\widetilde\delta^{-1}v\|_{(H^{1/2-\varepsilon}(\Omega))^\prime}
\end{align*}
Then, we compute
\begin{eqnarray*} 
\|\widetilde \delta_\Gamma^{-1} v\|_{(H^{1/2-\varepsilon}(\Omega))^\prime}
&=& 
\sup_{z \in H^{1/2-\varepsilon}} 
\frac{\langle \widetilde \delta_\Gamma^{-1/2} v,\widetilde \delta_\Gamma^{-1/2} z\rangle}
     {\|z\|_{H^{1/2-\varepsilon}(\Omega)}}
\leq \|\widetilde \delta_\Gamma^{-1/2} v\|_{L^2(\Omega)} h^{-\varepsilon} \sup_{z \in H^{1/2-\varepsilon}(\Omega)} 
\frac{\|\widetilde \delta_\Gamma^{-1/2+\varepsilon} z\|_{L^2(\Omega)}}{\|z\|_{H^{1/2-\varepsilon}(\Omega)}}. 
\end{eqnarray*}
Again, estimate~\eqref{eq:lemma:weighted-embedding-1} of Lemma~\ref{lemma:weighted-embedding} finishes the proof.
\end{proof}


\section{Numerical analysis.}
\label{sec:numerical-analysis}

\subsection{Main results.}
\label{sec:main-results}

In the following, we assume that the approximation space $V_h$, $M_h$ of \eqref{eq:weak-form-FEM} are 
spaces of piecewise polynomials. For future reference, we formulate their properties as an assumption. 

\begin{assumption}
\label{assumption:X_h}
Let ${\mathcal T}_\Omega$ and ${\mathcal T}_\Gamma$ be two (not necessarily matching)
quasi-uniform, affine triangulations of $\Omega$ and $\Gamma$ into 
volume and surfaces simplices (e.g., for $d = 3$ tetrahedra and surface
triangles)
both with mesh size $h$.  For a fixed $k \in \BbbN$, let 
$V_h:= S^{k,1}({\mathcal T}_\Omega):= 
\{v \in H^1(\Omega)\,|\, v|_K \in {\mathcal P}_k \quad \forall K \in {\mathcal T}_\Omega\}$, 
$M_h:= S^{k-1,0}({\mathcal T}_\Gamma):= 
\{v \in L^2(\Gamma)\,|\, v|_K \in {\mathcal P}_{k-1} \quad \forall K \in {\mathcal T}_\Gamma\}$
be spaces of piecewise polynomials of degree $k$ and $k-1$, respectively. Set $X_h:= V_h \times M_h$. 
\end{assumption}

\begin{remark}
{\normalfont
Although Assumption~\ref{assumption:X_h} does not require the meshes 
${\mathcal T}_\Omega$ and ${\mathcal T}_\Gamma$ to be matching, it is natural to do so
in implementations. The analysis of the following 
Theorem~\ref{thm:approximation-of-varphi}
can be generalized to the case of two quasi-uniform meshes ${\mathcal T}_\Omega$, ${\mathcal T}_\Gamma$
with differing mesh sizes $h_\Omega$, $h_\Gamma$. 
}\eremk
\end{remark}

Our starting point are the Galerkin orthogonalities satisfied by the 
exact solution $(u,\varphi) \in X$ and its Galerkin 
approximation $(u_h,\varphi_h) \in X_h$ that are obtained by 
subtracting \eqref{eq:weak-form-FEM} from \eqref{eq:weak-form}; in order to 
be able to account for certain types of variational crimes 
we include 
additionally two linear forms $\varepsilon_1:H^1(\Omega)\to\mathbb R$ and $\varepsilon_2:H^{-1/2}(\Gamma)\to\mathbb R$ 
on the right-hand side: 
\begin{subequations}
\label{eq:orthogonalities-primal}
\begin{eqnarray}
\label{eq:orthogonalities-primal-1}
\widetilde a(u - u_h,v) - b(v,\varphi - \varphi_h) &=& \varepsilon_1(v) \qquad \forall v \in V_h,  \\
\label{eq:orthogonalities-primal-2}
b(u - u_h,\psi) + c(\varphi - \varphi_h,\psi) &=& \varepsilon_2(\psi)  \qquad \forall \psi \in M_h . 
\end{eqnarray}
\end{subequations}

\begin{remark}
{\normalfont
The exact Galerkin orthogonalities have the above form~\eqref{eq:orthogonalities-primal} with 
$\varepsilon_1 \equiv 0$ and $\varepsilon_2 \equiv 0$. 
The terms $\varepsilon_1$ and $\varepsilon_2$ are appropriate to control additional errors 
introduced by approximating the jumps $u_0$ and $\phi_0$  (cf. \eqref{eq:rhs-weak}), e.g., by piecewise 
polynomial functions. Such approximations are practically unavoidable in view of the fact that 
the hypersingular operator appears on the right-hand side~\eqref{eq:rhs-weak}
of the coupling equations~\eqref{eq:weak-form}.
 This issue will be studied further
in Section~\ref{sec:randnaehe} and plays a role in the numerical examples in Section~\ref{sec:numerics}. 
}\eremk
\end{remark}

We recall that the standard convergence theory 
(cf.\ Lemma~\ref{lemma:solvability} (\ref{lemma:solvability:item:vi})) yields $O(h^k)$ 
under the regularity assumption $(u,\varphi) \in H^{k+1}(\Omega) \times H^{k-1/2}_{pw}(\Gamma)$:
\begin{align}
\|u - u_h\|_{H^1(\Omega)} + \|\varphi - \varphi_h\|_{H^{-1/2}(\Gamma)} 
&\leq C \inf_{(v,\psi) \in X_h} \left[\|u - v\|_{H^1(\Omega)} + \|\varphi - \psi\|_{H^{-1/2}(\Gamma)} \right]
\label{eq:standard-estimate-2} 
 \leq C h^k \left[ \|u\|_{H^{k+1}(\Omega)} + \|\varphi\|_{H^{k-1/2}_{pw}(\Gamma)}\right]. 
\end{align}
The bound (\ref{eq:standard-estimate-2}) implies $\|u - u_h\|_{H^1(\Omega)}  = O(h^k)$, which is the best
rate achievable when approximating $u$ with piecewise polynomials of degree $k$. We observe, however, that
the approximation results for the two contributions $\inf_{v \in V_h} \|u - v\|_{H^1(\Omega)}$ and 
$\inf_{\psi \in M_h} \|\varphi - \psi\|_{H^{-1/2}(\Gamma)}$ are imbalanced: the regularity assumption
$\varphi \in H^{k}_{pw}(\Gamma)$ 
and the approximation properties of $M_h$ 
imply even $\inf_{\psi \in M_h} \|\varphi - \psi\|_{H^{-1/2}(\Gamma)}
\leq C h^{k+1/2} \|\varphi\|_{H^{k}_{pw}(\Gamma)}$. The {\em joint} approximation of $u$ and $\varphi$
in (\ref{eq:standard-estimate-2}) cannot exploit this. However, given additional regularity of 
$u$ and $\varphi$, this optimal rate $O(h^{k+1/2})$ for $\|\varphi - \varphi_h\|_{H^{-1/2}(\Gamma)}$ 
can be achieved as we now show in the following Theorem~\ref{thm:approximation-of-varphi}. 

%
\begin{theorem}
\label{thm:approximation-of-varphi}
Let $\opA \in C^{0,1}(\overline{\Omega})$ if $k=1$ and 
$\opA \in C^{1,1}(\overline{\Omega})$ if $k>1$. 
Let Assumption~\ref{assumption:shift-theorem} and~\ref{assumption:shift-theorem-bidual-problem} 
be valid\footnote{Recall that these two assumptions coincide in the present case.}. 
Let $X_h = V_h \times M_h$ be given by Assumption~\ref{assumption:X_h}. 
Let $(u,\varphi) \in X$ be the solution of \eqref{eq:rhs-weak}--\eqref{eq:weak-form}
and $(u_h,\varphi_h) \in X_h$ be the solution of \eqref{eq:weak-form-FEM}. 
\begin{enumerate}[(i)]
\item 
\label{item:thm:approximation-of-varphi-i}
The error bound (\ref{eq:standard-estimate-2}) holds, 
if $(u,\varphi) \in H^{k+1}(\Omega) \times H^{k-1/2}_{pw}(\Gamma)$. The constant $C>0$ in (\ref{eq:standard-estimate-2})
depends only on $\Omega$, the coercivity constant $\alpha_0$ of $\opA$, 
the upper bound $\|\opA\|_{L^\infty(\Omega)}$, the approximation order $k$, and the shape regularity of 
the quasi-uniform triangulations $\mathcal T_\Omega$, $\mathcal T_\Gamma$.
\item  
\label{item:thm:approximation-of-varphi-ii}
Suppose extra regularity
$u \in B^{k+3/2}_{2,1}(\Omega)$ and 
$\varphi \in H^{k}_{pw}(\Gamma)$.
Then, we have
\begin{align}
\label{eq:thm1}
\|\varphi - \varphi_h\|_{H^{-1/2}(\Gamma)} 
&\leq C h^{k+1/2} (1 + \delta_{k,1} |\ln h|) 
\|u\|_{B^{k+3/2}_{2,1}(\Omega)} + C h^{k+1/2} \|\varphi\|_{H^{k}_{pw}(\Gamma)},\\
\label{eq:thm2}
\|u - u_h\|_{L^2(S_h)} 
&\leq C h^{3/2+k} (1 + \delta_{k,1} |\ln h|) \|u\|_{B^{k+3/2}_{2,1}(\Omega)} 
+ C h^{3/2+k} \|\varphi\|_{H^{k}_{pw}(\Gamma)}. 
\end{align}
Here, $\delta_{k,1}$ denotes the Kronecker symbol, i.e., $\delta_{1,1} = 1$ and $\delta_{k,1} = 0$ for $k \ne 1$.
The constant $C>0$ depends only on $\Omega$, the coefficient $\opA$, the approximation order $k$, 
Assumptions~\ref{assumption:shift-theorem} and~\ref{assumption:shift-theorem-bidual-problem}, 
as well as shape regularity of the quasi-uniform triangulations 
$\mathcal T_\Omega$ and $\mathcal T_\Gamma$. More precisely, the dependence on $\opA$ is--in addition to 
Assumptions~\ref{assumption:shift-theorem}, \ref{assumption:shift-theorem-bidual-problem}--in terms of 
the coercivity constant $\alpha_0$ of $\opA$, the bound 
$\|\opA\|_{C^{0,1}(\overline{\Omega})}$ for $k=1$ and 
$\|\opA\|_{C^{1,1}(\overline{\Omega})}$ for $k>1$. 
\end{enumerate}
\end{theorem}%


\subsection{Approximation estimates and proof of estimate~\eqref{eq:thm1} of Theorem~\protect{\ref{thm:approximation-of-varphi}}}
\label{section:approx}%

We recall that the spaces $V_h$ and $M_h$ have the following approximation
properties:

\begin{lemma}
\label{lemma:approximation-properties-Vh-Mh} 
\begin{enumerate}[(i)]
\item 
\label{item:lemma:approximation-properties-Vh-Mh-1} 
There is an elementwise defined (nodal) interpolation operator $\Ihk:C(\overline{\Omega}) \rightarrow V_h$
such that for integers $j\ge 0$, $\ell\ge 1$ with $0 \leq j \leq \ell+1 \leq k+1$ and every $K \in {\mathcal T}_\Omega$ and 
sufficiently smooth $u$
\begin{align*}
\|\nabla^j( u - \Ihk u)\|_{L^2(K)} &\leq C \,\diam{(K)}^{\ell-j+1} \|\nabla^{\ell+1} u\|_{L^2(K)}; \\
\|\nabla_f^j( u - \Ihk u)\|_{L^2(f)} &\leq C \,\diam{(f)}^{\ell-j+1} \|\nabla^{\ell+1}_f u\|_{L^2(f)}
\qquad \mbox{ for all faces $f \subset \partial K$;}
\end{align*}
here, $\nabla_f$ represents the surface gradient on the face $f$. 
\item 
\label{item:lemma:approximation-properties-Vh-Mh-2} 
For $\varepsilon \ge 0$ and fixed $0 < D < D^\prime $ 
$$
C^{-1}\,\|\widetilde \delta_\Gamma^{-1/2-\varepsilon} \nabla (u - \Ihk u)\|_{L^2(S_D)} 
\le h^k \|\widetilde\delta_\Gamma^{-1/2-\varepsilon} \nabla^{k+1} u\|_{L^2(S_{D+h})}
\le C\, h^k \|\nabla^{k+1} u\|_{B^{1/2}_{2,1}(S_{D^\prime})}
\begin{cases} |\ln h|^{1/2} , & \mbox{ if $\varepsilon = 0$}, \\
              h^{-\varepsilon} , & \mbox{ if $\varepsilon > 0$},
\end{cases}
$$ 
\item 
\label{item:lemma:approximation-properties-Vh-Mh-3} 
There are bounded linear operators $R_h:B^{3/2}_{2,\infty}(\Omega) \rightarrow V_h$ and 
$Q_h:B^0_{2,\infty}(\Gamma) \rightarrow M_h$ with 
$$
\|w - R_h w\|_{H^1(\Omega)} \leq C h^{1/2} \|w\|_{B^{3/2}_{2,\infty}(\Omega)}, 
\qquad 
\|\psi - Q_h \psi\|_{V} \leq C h^{1/2} \|\psi\|_{B^{0}_{2,\infty}(\Gamma)}. 
$$
\end{enumerate}%
The constant $C>0$ depends only on $k$ and the shape regularity of $\mathcal T_\Omega$ and $\mathcal T_\Gamma$,
respectively.
\end{lemma}

\begin{proof}
The assertion 
\eqref{item:lemma:approximation-properties-Vh-Mh-1} is well-known. For 
\eqref{item:lemma:approximation-properties-Vh-Mh-2}, we note that 
\eqref{item:lemma:approximation-properties-Vh-Mh-1} yields the first inequality. The second 
inequality follows from estimates~\eqref{eq:lemma:weighted-embedding-2}--\eqref{eq:lemma:weighted-embedding-4} 
of Lemma~\ref{lemma:weighted-embedding}. In 
\eqref{item:lemma:approximation-properties-Vh-Mh-3}, we only show the construction of $Q_h$. 
It suffices to consider
the lowest order case $k=1$,
i.e., $M_h$ consists of piecewise constant functions. For simplicity,
let $Q_h$ be the projection in the $H^{-1/2}(\Gamma)$-inner product. Then for (fixed)
$\varepsilon \in (0,1/2)$ by standard
approximation properties
$\|\operatorname*{I} - Q_h\|_{H^{-1/2}(\Gamma) \leftarrow H^{-\varepsilon}(\Gamma)} \leq h^{1/2-\varepsilon}$
and
$\|\operatorname*{I} - Q_h\|_{H^{-1/2}(\Gamma) \leftarrow H^{\varepsilon}(\Gamma)} \leq h^{1/2+\varepsilon}$.
The result follows by interpolation.
\end{proof}

The proof of estimate~\eqref{eq:thm2} in Theorem~\ref{thm:approximation-of-varphi} is postponed to 
Section~\ref{sec:proof-of-thm:error-on-strip} since it requires further 
auxiliary results, which are provided in Section~\ref{section:blubber}--\ref{subsection:dual}. 
The estimate~\eqref{eq:thm1}, however, is an immediate consequence 
of~\eqref{eq:thm2} as we now show.

\begin{numberedproof}{Theorem~\protect{\ref{thm:approximation-of-varphi}, equation~(\ref{eq:thm1})}}
We suppose that~\eqref{eq:thm2} is valid.
The proof of~\eqref{eq:thm1} is based on the Galerkin orthogonality 
\eqref{eq:orthogonalities-primal-2} (with $\varepsilon_2 \equiv 0$ there). We have 
for arbitrary $\IV u \in V_h$ and $\IM \varphi \in M_h$ (these two elements will be chosen suitably 
below) 
\begin{eqnarray*}
\|\IM \varphi - \varphi_h\|^2_V &= & 
c(\IM \varphi -\varphi_h,\IM \varphi - \varphi_h)  \\
&=& 
c(\IM \varphi -\varphi,\IM \varphi - \varphi_h)  + 
c(\varphi -\varphi_h,\IM \varphi - \varphi_h)  \\
&\stackrel{\eqref{eq:orthogonalities-primal-2}}=& 
c(\IM \varphi -\varphi,\IM \varphi - \varphi_h)  - b(u - u_h,\IM \varphi - \varphi_h)
\\
&\lesssim &
\|\IM \varphi - \varphi\|_{V} \|\IM \varphi - \varphi_h\|_{V} 
+ \|u - u_h\|_{H^{1/2}(\Gamma)} \|\IM \varphi - \varphi_h\|_{V}. 
\end{eqnarray*}
Hence, 
$\displaystyle 
\|\IM \varphi - \varphi_h\|_{V} \lesssim  \|\varphi - \IM \varphi \|_{V} + 
\|u - u_h\|_{H^{1/2}(\Gamma)}.
$
The term $\|u - u_h\|_{H^{1/2}(\Gamma)}$ is estimated with 
an inverse estimate on $\Gamma$ and a suitable norm equivalence
as follows: 
\begin{eqnarray*}
\|u - u_h\|_{H^{1/2}(\Gamma)} &\leq& 
\|u - \IV u\|_{H^{1/2}(\Gamma)} + 
\|\IV u - u_h\|_{H^{1/2}(\Gamma)} 
\\
&\lesssim& 
\|u - \IV u\|_{H^{1/2}(\Gamma)} + 
h^{-1/2} \|\IV u - u_h\|_{L^{2}(\Gamma)}  \\
& \lesssim &
\|u - \IV u\|_{H^{1/2}(\Gamma)} + 
h^{-1} \|\IV u - u_h\|_{L^{2}(S_h)}. 
\end{eqnarray*}
Next,
we use continuity of the 
trace operator $\gamma:B^{1/2}_{2,1}(\Omega) \rightarrow L^2(\Gamma)$
(see, e.g., \cite[Thm.~{2.9.3}]{triebel95}) to get boundedness of 
$\gamma: B^{k+3/2}_{2,1}(\Omega) \rightarrow H^{k+1}_{pw}(\Gamma) \cap H^1(\Gamma)$. 
Selecting $\IV u = \Ihk u \in V_h $ as the nodal interpolant of 
Lemma~\ref{lemma:approximation-properties-Vh-Mh},
we get 
\begin{eqnarray}
\label{eq:foo-1}
\|u - \Ihk u \|_{L^{2}(\Gamma)} 
+ h \|u - \Ihk u \|_{H^{1}(\Gamma)} &\lesssim& h^{k+1} \|\gamma u\|_{H^{k+1}_{pw}(\Gamma)}, \\
\label{eq:foo-2}
\|u-\Ihk u\|_{L^2(S_h)} &\lesssim & h^{k+1} \|\nabla^{k+1} u\|_{L^{2}(S_{2h})}. 
\end{eqnarray}
An interpolation argument in combination with (\ref{eq:foo-1}) 
and $\|u\|_{H^{k+1}_{pw}(\Gamma)} \lesssim \|u\|_{B^{k+3/2}_{2,1}(\Omega)}$ gives 
\begin{equation}
\label{eq:foo-4}
\|u - \Ihk u\|_{H^{1/2}(\Gamma)} \lesssim h^{k+1/2} \|\gamma u\|_{H^{k+1}_{pw}(\Gamma)} 
\lesssim h^{k+1/2} \|u\|_{B^{k+3/2}_{2,1}(\Omega)}.
\end{equation}
In order to estimate (\ref{eq:foo-2}), we observe
that \cite[Lemma~{2.1}]{li-melenk-wohlmuth-zou10} gives 
$\|\nabla^{k+1} u\|_{L^2(S_{2h})} \lesssim h^{1/2} \|\nabla^{k+1} u\|_{B^{1/2}_{2,1}(\Omega)} 
\lesssim h^{1/2} \|u\|_{B^{k+3/2}_{2,1}(\Omega)}$. In total, we arrive at 
\begin{equation}
\label{eq:foo-3}
\|u - \Ihk u \|_{H^{1/2}(\Gamma)} \lesssim  h^{k+1/2} \|u\|_{B^{k+3/2}_{2,1}(\Omega)} 
\qquad \mbox{ and } \qquad 
\|u-\Ihk u\|_{L^2(S_h)} \lesssim  h^{k+3/2} \|u\|_{B^{k+3/2}_{2,1}(\Omega)}. 
\end{equation}
The approximation properties of $M_h$ yield the existence of 
$\IM \varphi \in M_h$ with $\|\varphi - \IM \varphi\|_{V} 
\lesssim h^{k+1/2} \|\varphi\|_{H^k_{pw}(\Gamma)}$. Combining these estimates with~\eqref{eq:thm2} yields 
\begin{eqnarray*}
\|\varphi - \varphi_h\|_{V} \lesssim 
\|\varphi - \IM \varphi \|_{V} + \|u - \Ihk u\|_{H^{1/2}(\Gamma)} 
+ h^{-1} \|u - \Ihk u\|_{L^2(S_h)} 
+ h^{-1} \|u - u_h\|_{L^2(S_h)}
\lesssim h^{k+1/2} \Big[ \|\varphi\|_{H^k_{pw}(\Gamma)}  + \|u\|_{B^{k+3/2}_{2,1}(\Omega)}\Big]. 
\end{eqnarray*}
This concludes the proof.
\end{numberedproof}

\subsection{Local estimates via duality arguments.}\label{section:blubber}
For the solution $(u,\varphi) \in X$ of \eqref{eq:weak-form}
and its Galerkin approximation $(u_h,\varphi_h) \in X_h$, which solves  \eqref{eq:weak-form-FEM}, 
we define the error 
\begin{equation}
\label{eq:e}
e:= u - u_h. 
\end{equation}
Take the cut-off function $\cutoff$ to be the characteristic function of $S_h$. 
Let $(w,\lambda) = T^\dual(\cutoff e)$ be the solution of the
dual problem
\begin{subequations}
\label{eq:dual-problem-concrete}
\begin{eqnarray} 
\label{eq:dual-problem-concrete-1}
\widetilde a(v,w) - b(v,\lambda) &=& \langle v,\cutoff e\rangle_\Omega 
\qquad \forall v \in H^1(\Omega), \\
\label{eq:dual-problem-concrete-2}
b(w,\psi) + c(\psi,\lambda) &=& 0 \qquad \forall \psi \in H^{-1/2}(\Gamma). 
\end{eqnarray} 
\end{subequations}
Its Galerkin approximation $(w_h,\lambda_h) \in X_h$ is given by 
\begin{subequations}
\label{eq:dual-problem-fem}
\begin{eqnarray} 
\label{eq:dual-problem-fem-1}
\widetilde a(v,w_h) - b(v,\lambda_h) &=& \langle v,\cutoff e\rangle_\Omega 
\qquad \forall v \in V_h, \\
\label{eq:dual-problem-fem-2}
b(w_h,\psi) + c(\psi,\lambda_h) &=& 0 \qquad \forall \psi \in M_h. 
\end{eqnarray} 
\end{subequations}
Subtracting 
\eqref{eq:dual-problem-fem} from 
\eqref{eq:dual-problem-concrete} leads to the 
following Galerkin orthogonalities:
\begin{subequations}
\label{eq:orthogonalities-dual}
\begin{eqnarray}
\label{eq:orthogonalities-dual-1}
\widetilde a(v,w - w_h) - b(v,\lambda - \lambda_h) &=& 0 \qquad \forall v \in V_h,  \\
\label{eq:orthogonalities-dual-2}
b(w - w_h,\psi) + c(\psi,\lambda - \lambda_h) &=& 0  
\qquad \forall \psi \in M_h.
\end{eqnarray}
\end{subequations}

\begin{lemma}
\label{lemma:error-on-strip}
For arbitrary pair $(\IV u,\IM \varphi) \in X_h$, we have 
\begin{eqnarray*}
\|\cutoff e\|_{L^2(\Omega)}^2 =
\widetilde a(u - \IV u,w - w_h) - b(u - \IV u,\lambda - \lambda_h) - 
b(w - w_h,\varphi - \IM \varphi) - c(\varphi - \IM \varphi,\lambda - \lambda_h) 
+ \varepsilon_1(w_h) - \varepsilon_2(\lambda_h). 
\end{eqnarray*}
\end{lemma}

\begin{proof} The proof follows from simple manipulations with 
the Galerkin orthogonalities~\eqref{eq:orthogonalities-primal}
and~\eqref{eq:orthogonalities-dual-1} and the defining equations: 
\begin{eqnarray*}
\langle \cutoff e,e\rangle_\Omega & \stackrel{\eqref{eq:dual-problem-concrete-1}}{=} & 
\widetilde a(e,w) - b(e,\lambda)  
\\ &=& 
\widetilde a(e,w - w_h) + \widetilde a(e,w_h) 
- b(e,\lambda - \lambda_h) - b(e,\lambda_h)  
\\ &=& 
\widetilde a(u - \IV u,w - w_h) + \widetilde a(\IV u - u_h,w - w_h)
+ \widetilde a(e,w_h) - b(u - \IV u,\lambda - \lambda_h) 
- b(\IV u - u_h,\lambda - \lambda_h) - b(e,\lambda_h)
\\ 
& \stackrel{\eqref{eq:orthogonalities-dual-1}}{=} &
\widetilde a(u - \IV u,w - w_h) - b(u - \IV u,\lambda - \lambda_h) 
+ \underbrace{\widetilde a(e,w_h)}_{=:I} - \underbrace{b(e,\lambda_h)}_{=:II}.  
\end{eqnarray*}
We rearrange the terms $I$ and $II$. 
\begin{eqnarray*}
I = \widetilde a(e,w_h) &\stackrel{\eqref{eq:orthogonalities-primal-1}}{=}&
b(w_h,\varphi - \varphi_h)  
+\varepsilon_1(w_h) 
\\
&=& b(w_h,\varphi - \IM \varphi) + b(w_h,\IM \varphi - \varphi_h)  
+ \varepsilon_1(w_h)\\
&=& b(w_h - w ,\varphi - \IM \varphi) + b(w,\varphi - \IM \varphi) + b(w_h,\IM \varphi - \varphi_h) 
+\varepsilon_1(w_h)\\ 
&\stackrel{\eqref{eq:dual-problem-concrete-2},\eqref{eq:dual-problem-fem-2}}{=} & b(w_h - w,\varphi - \IM \varphi) - c(\varphi - \IM \varphi,\lambda) 
-  c(\IM \varphi - \varphi_h,\lambda_h)  
+ \varepsilon_1(w_h) \\
II = b(e,\lambda_h)  
&\stackrel{\eqref{eq:orthogonalities-primal-2}}{=}& 
- c(\varphi - \varphi_h,\lambda_h)  + \varepsilon_2(\lambda_h). 
\end{eqnarray*}
Hence, we obtain 
\begin{align*}
\langle \cutoff e,e\rangle_\Omega &= 
\widetilde a(u - \IV u,w - w_h) - b(u - \IV u,\lambda - \lambda_h) 
+ I - II \\
&= 
\widetilde a(u - \IV u,w - w_h) - b(u - \IV u,\lambda - \lambda_h) 
+ b(w_h - w,\varphi-  \IM \varphi)\\
&\qquad \mbox{}   - c(\varphi - \IM \varphi,\lambda)
- c(\IM \varphi - \varphi_h,\lambda_h) + c(\varphi - \varphi_h,\lambda_h) 
+ \varepsilon_1(w_h) - \varepsilon_2(\lambda_h)
\\
&= 
\widetilde a(u - \IV u,w - w_h) - b(u - \IV u,\lambda - \lambda_h) 
+ b(w_h - w,\varphi-  \IM \varphi)  
- c(\varphi - \IM \varphi ,\lambda - \lambda_h)  + \varepsilon_1(w_h) - \varepsilon_2(\lambda_h), 
\end{align*}
which is the desired equality.
\end{proof}

\subsection{Analysis of the dual problems: estimating $w - w_h$ and $\lambda - \lambda_h$.}\label{subsection:dual}

Lemma~\ref{lemma:error-on-strip} shows that we can infer bounds 
for the error $u - u_h$ on a strip $S_h$ near $\Gamma$ from knowledge 
about the errors $w - w_h$ and $\lambda - \lambda_h$. The additional
two terms $\varepsilon_1(w_h)$ and $\varepsilon_2(\lambda_h)$ that
appear in Lemma~\ref{lemma:error-on-strip} will be bounded 
in Theorem~\ref{thm:variational-crimes} below; we recall that they
were introduced to treat certain types of variational crimes. 

We will need the following regularity assertions for the 
solution $(w,\lambda) = T^\dual (\cutoff e)$ of the dual
problem~\eqref{eq:dual-problem}, which follow from 
Lemma~\ref{lemma:B32-regularity}:
\begin{subequations}
\label{eq:regularity-of-w}
\begin{eqnarray}
\label{eq:regularity-of-w-1}
\|w\|_{B^{3/2}_{2,\infty}(\Omega)} + \|\lambda\|_{B^0_{2,\infty}(\Gamma)} 
&\lesssim& h^{1/2} \|\cutoff e\|_{L^2(\Omega)} , \\
\label{eq:regularity-of-w-2}
\|w\|_{H^{3/2+\varepsilon}(\Omega)} + \|\lambda\|_{H^\varepsilon_{pw}(\Gamma)} 
&\lesssim& h^{1/2-\varepsilon} \|\cutoff e\|_{L^2(\Omega)}  \quad 
\forall \,0<\varepsilon\le s_0-1/2.
\end{eqnarray}
\end{subequations}

\subsubsection{Error analysis of $w - w_h$ and $\lambda - \lambda_h$ in the energy norms.}

The uniform 
inf-sup stability of the bilinear form $A(\cdot,\cdot)$ (cf.~Lemma~\ref{lemma:solvability}) 
provides the following {\sl a priori} bound: 

\begin{lemma}
\label{lemma:a-priori-dual}
Let Assumption~\ref{assumption:shift-theorem} be valid. Then,
$\|w - w_h\|_{H^1(\Omega)} + \|\lambda - \lambda_h\|_{H^{-1/2}(\Gamma)} 
\leq C h \|\cutoff e\|_{L^2(\Omega)}. $
The constant $C>0$ depends only on $\Omega$, Assumption~\ref{assumption:shift-theorem},
the shape regularity of ${\mathcal T}_\Omega$, ${\mathcal T}_\Gamma$, 
and $\opA$ through the coercivity constant of $\opA$ and $\|\opA\|_{L^\infty(\Omega)}$.
\end{lemma}

\begin{proof}
As observed in Section~\ref{sec:dual}, the dual 
problem corresponds to a transmission problem and is hence covered by 
Lemma~\ref{lemma:solvability}.
By the uniform inf-sup stability ascertained in
Lemma~\ref{lemma:solvability}, \eqref{lemma:solvability:item:vi}, we have the 
quasi-optimality~\eqref{eq:quasi-optimality}.
Combining the 
regularity assertion \eqref{eq:regularity-of-w-1} with the approximation
properties of 
Lemma~\ref{lemma:approximation-properties-Vh-Mh} gives 
\begin{align*}
\|w - w_h\|_{H^1(\Omega)} + \|\lambda - \lambda_h\|_{H^{-1/2}(\Gamma)} 
\leq C \inf_{(v,\mu) \in X_h} 
\Big[\|w - v\|_{H^1(\Omega)} + \|\lambda - \mu\|_{H^{-1/2}(\Gamma)} \Big]
&\leq C h^{1/2} \Big[ \|w\|_{B^{3/2}_{2,\infty}(\Omega)} 
                      + \|\lambda\|_{B^{0}_{2,\infty}(\Gamma)}\Big]
\\&\leq C h^{1/2} h^{1/2} \|\cutoff e\|_{L^2(\Omega)}. 
\end{align*}
This concludes the proof.
\end{proof}

The estimates of Lemma~\ref{lemma:a-priori-dual}
allow us to control the terms
$b(u - \IV u,\lambda - \lambda_h)$,
$b(w - w_h,\varphi- \IM \varphi)$,
$c(\varphi - \IM \varphi,\lambda- \lambda_h)$
that appear in Lemma~\ref{lemma:error-on-strip}: 
\begin{eqnarray}
\label{eq:estimate-most-terms-of-lemma:a-priori-dual}
\left| b(u - \IV u,\lambda - \lambda_h) 
\right| + 
\left|
b(w - w_h,\varphi- \IM \varphi)
\right| + 
\left| c(\varphi - \IM \varphi,\lambda-  \lambda_h)
\right| 
\\\qquad\leq C h \|\cutoff e\|_{L^2(\Omega)} 
\Big[ \|u - \IV u\|_{H^{1/2}(\Gamma)} + 
       \|\varphi - \IM \varphi\|_{H^{-1/2}(\Gamma)} 
\Big].
\end{eqnarray}
Moreover, Lemma~\ref{lemma:a-priori-dual} yields the estimate
\begin{equation}
\label{eq:estimate-part-of-tilde-a-of-lemma:a-priori-dual}
\left|\langle \hyp (u - \IV u),w - w_h\rangle_\Gamma\right| 
\leq C h \|\cutoff e\|_{L^2(\Omega)} \|u - \IV u\|_{H^{1/2}(\Gamma)} ,
\end{equation}
which is a part of the term $\widetilde a(u - \IV u,w - w_h)$ in Lemma~\ref{lemma:error-on-strip}.
Thus, from the terms appearing in Lemma~\ref{lemma:error-on-strip}, only the term 
$a(u - \IV u,w - w_h)$ remains to be controlled. Its analysis is more elaborate and
requires an analysis of $w - w_h$ in weighted norms. To see how this comes about, 
we fix $D > 0$ and write with a parameter $\varepsilon \ge 0$ that will be selected later
\begin{align}
\label{eq:cor:randnaehe-10}
|a(u - \IV u,w - w_h)|  &=  \left| 
\int_{S_{D}} \opA \nabla (u - \IV u) \cdot \nabla (w - w_h) 
+ 
\int_{\Omega\setminus S_{D}} \opA \nabla (u - \IV u) \cdot \nabla (w - w_h) 
\right| 
\\
\nonumber 
&\lesssim 
\|\widetilde\delta_\Gamma^{-1/2-\varepsilon} \nabla (u - \IV u)\|_{L^2(S_D)} 
\|\widetilde\delta_\Gamma^{1/2+\varepsilon} \nabla (w - w_h)\|_{L^2(S_D)} 
+ \|\nabla(u - \IV u)\|_{L^2(\Omega\setminus S_{D})}
 \|\nabla(w - w_h)\|_{L^2(\Omega\setminus S_{D})}. 
\end{align} 
The choice $\IV u  = \Ihk u$ with the nodal interpolant $\Ihk u$ of 
Lemma~\ref{lemma:approximation-properties-Vh-Mh} puts us on familiar ground 
for the factors 
$\|\widetilde\delta_\Gamma^{-1/2-\varepsilon} \nabla (u - \IV u)\|_{L^2(S_D)}$ 
and  
$\|\nabla (u - \IV u)\|_{L^2(S_D)}$. 
Hence,
we are left with estimating  
$\|\widetilde\delta_\Gamma^{1/2+\varepsilon} \nabla (w - w_h)\|_{L^2(S_D)}$ and 
$\|\nabla (w - w_h)\|_{L^2(\Omega\setminus S_{D})}$, which is achieved in the
subsequent Section~\ref{sec:w-w_h-nonstandard}. Although the parameter $\varepsilon \ge 0$ 
is arbitrary at this point, we mention that we will select $\varepsilon > 0$ arbitrary (but small)
for the case $k > 1$ and $\varepsilon = 0$ for the lowest order case $k = 1$. 

\subsubsection{Error analysis of $w - w_h$ and $\lambda - \lambda_h$ in weighted norms.}
\label{sec:w-w_h-nonstandard}

We estimate $\|\widetilde\delta_\Gamma^{1/2+\varepsilon} \nabla (w - w_h)\|_{L^2(S_D)}$ in a manner that is 
structurally similar to the procedure in \cite{melenk-wohlmuth14a} and also \cite[Sec.~{5.1.2}]{melenk-wohlmuth12}. 
Basically, we employ tools from local error analysis
of FEM as described, for example, in \cite[Sec.~{5.3}]{wahlbin95} to control $w - w_h$ in terms 
of a best approximation in a weighted $H^1$-norm and a lower-order term in a weighted $L^2$-norm. 
The best approximation in a weighted $H^1$-norm is estimated with a standard nodal interpolant; 
the lower-order term in a weighted $L^2$-norm requires more care and is handled in the following 
Lemma~\ref{lemma:weighted-w-wh}. 

\begin{lemma}
\label{lemma:weighted-w-wh}
Let Assumptions~\ref{assumption:shift-theorem} and \ref{assumption:shift-theorem-bidual-problem}
be valid. 
With the regularized distance function $\widetilde \delta_\Gamma = \delta_\Gamma +h$ from~\eqref{eq:strip-delta} we have 
for $0 < \varepsilon \leq s_0-1/2$ 
\begin{eqnarray}
\label{eq:lemma:weighted-w-wh-1}
\|\widetilde \delta_\Gamma ^{-1/2} (w - w_h)\|_{L^2(\Omega)} &\leq& 
C_1 h^{3/2} |\ln h|^{1/2} \|\cutoff e\|_{L^2(\Omega)}, \\
\label{eq:lemma:weighted-w-wh-2}
\|\widetilde \delta_\Gamma^{-1/2+\varepsilon} (w - w_h)\|_{L^2(\Omega)} &\leq& 
C_2 h^{3/2+\varepsilon}  \|\cutoff e\|_{L^2(\Omega)}. 
\end{eqnarray}
The constant $C_1>0$ depends on the same quantities as the constant in Lemma~\ref{lemma:a-priori-dual}
and additionally on Assumption~\ref{assumption:shift-theorem-bidual-problem}. The constant $C_2>0$
depends furthermore on $\varepsilon$. 
\end{lemma}

\begin{proof}
Both estimates require yet another duality argument. 

{\em Proof of \eqref{eq:lemma:weighted-w-wh-1}:}
Abbreviate
$\displaystyle 
e_w:= w - w_h
$
and let
$(z,\psi) := T^\bidual (\widetilde \delta_\Gamma^{-1} e_w)$
denote the solution of the bidual problem
\begin{subequations}
\label{eq:dual-dual}
\begin{eqnarray}
\label{eq:dual-dual-1}
\widetilde a(z,v) - b(v,\psi) &=& \langle \widetilde\delta_\Gamma^{-1} e_w,v\rangle_\Omega \qquad \forall v \in H^1(\Omega), \\
\label{eq:dual-dual-2}
b(z,\mu) + c(\psi,\mu) &=& 0 \qquad \forall \mu \in H^{-1/2}(\Gamma), 
\end{eqnarray}
\end{subequations}
{}From this, we get with $v = e_w$ for arbitrary $\IV z \in V_h$ and $\IM \psi \in M_h$
\begin{eqnarray*}
\|\widetilde \delta_\Gamma^{-1/2} e_w\|^2_{L^2(\Omega)} 
&\stackrel{\eqref{eq:dual-dual-1}}{=}& 
\widetilde a(z,e_w) - b(e_w,\psi) \\
&=&
\widetilde a(z - \IV z,e_w) + \widetilde a(\IV z,e_w) - b(e_w,\psi-\IM \psi) -b(e_w,\IM \psi)\\
&\stackrel{\eqref{eq:orthogonalities-dual-1},\eqref{eq:orthogonalities-dual-2}}{=}&  
\widetilde a(z - \IV z,e_w) + b(\IV z,\lambda - \lambda_h)- b(e_w,\psi - \IM \psi)  
+ c(\IV \psi,\lambda - \lambda_h) \\
&=&  \widetilde a(z - \IV z,e_w) + b(\IV z-z,\lambda - \lambda_h) - b(e_w,\psi - \IM \psi)  
+ c(\IM \psi-\psi,\lambda - \lambda_h) \\
&& \qquad 
+ b(z,\lambda - \lambda_h) + c(\psi,\lambda - \lambda_h) \\
&\stackrel{\eqref{eq:dual-dual-2}} {=}& 
\widetilde a(z - \IV z,e_w) + b(\IV z-z,\lambda - \lambda_h) - b(e_w,\psi - \IM \psi)  
 + c(\IM \psi-\psi,\lambda - \lambda_h).  
\end{eqnarray*}
This yields
\begin{align*}
 \|\widetilde \delta_\Gamma^{-1/2} e_w\|^2_{L^2(\Omega)}
 \lesssim\Big[
 \|z-\IV z\|_{H^1(\Omega)} + \|\psi-\IM \psi\|_{H^{-1/2}(\Gamma)}
 \Big]
 \Big[
 \|w-w_h\|_{H^1(\Omega)} + \|\lambda-\lambda_h\|_{H^{-1/2}(\Gamma)} 
 \Big].
\end{align*}
By virtue of Assumption~\ref{assumption:shift-theorem-bidual-problem} and hence 
the \emph{a~priori} estimate~\eqref{eq:lemma:regularity-weighted-rhs-vorn-1}
of Lemma~\ref{lemma:regularity-weighted-rhs-vorn}, we have 
\begin{align*}
 (z,\psi) = T^{\rm bidual}(\widetilde\delta_\Gamma^{-1}e_w) 
 \in B_{2,\infty}^{3/2}(\Omega) \times B_{2,\infty}^0(\Gamma)
 \quad\text{with}\quad
 \|z\|_{B_{2,\infty}^{3/2}(\Omega)}
 + \|\psi\|_{B_{2,\infty}^0(\Gamma)}
 \lesssim |\ln h|^{1/2}\|\widetilde\delta_\Gamma^{-1/2}e_w\|_{L^2(\Omega)}.
\end{align*}
Lemma~\ref{lemma:approximation-properties-Vh-Mh}, (\ref{item:lemma:approximation-properties-Vh-Mh-3})
yields
\begin{align*}
 \inf_{(\IV z,\IM \psi)\in X_h}\Big[
 \|z-\IV z\|_{H^1(\Omega)} + \|\psi-\IM \psi\|_{H^{-1/2}(\Gamma)}
 \Big]
 \lesssim h^{1/2}\Big[
  \|z\|_{B_{2,\infty}^{3/2}(\Omega)}
 + \|\psi\|_{B_{2,\infty}^0(\Gamma)}
 \Big]
 \lesssim h^{1/2}|\ln h|^{1/2}\,\|\widetilde\delta_\Gamma^{-1/2}e_w\|_{L^2(\Omega)}.
\end{align*}
By virtue of Assumption~\ref{assumption:shift-theorem} and hence 
the {\sl a priori} estimate~\eqref{eq:lemma:B32-regularity-2}
of Lemma~\ref{lemma:B32-regularity}, we have 
\begin{align*}
 (w,\lambda)=T^{\rm dual}(\chi_{S_h}e)
  \in B_{2,\infty}^{3/2}(\Omega) \times B_{2,\infty}^0(\Gamma)
 \quad\text{with}\quad
 \|w\|_{B_{2,\infty}^{3/2}(\Omega)}
 + \|\lambda\|_{B_{2,\infty}^0(\Gamma)}
 \lesssim h^{1/2}\,\|\chi_{S_h}e\|_{L^2(\Omega)}.
\end{align*}
With the quasi-optimality~\eqref{eq:quasi-optimality} for $(w,\lambda)$,
we infer
\begin{align*}
 \|w-w_h\|_{H^1(\Omega)} + \|\lambda-\lambda_h\|_{H^{-1/2}(\Gamma)} 
 &\lesssim
 \inf_{(\IV w,\IM \lambda)\in X_h}\Big[
 \|w-\IV w\|_{H^1(\Omega)} + \|\lambda-\IM \lambda\|_{H^{-1/2}(\Gamma)}
 \Big]\\
 &\lesssim h^{1/2}\Big[\|w\|_{B_{2,\infty}^{3/2}(\Omega)}
 + \|\lambda\|_{B_{2,\infty}^0(\Gamma)}
\Big]
 \lesssim h\,\|\chi_{S_h}e\|_{L^2(\Omega)}.
\end{align*}
Altogether, we arrive at 
$$
\|\widetilde \delta_\Gamma^{-1/2} e_w\|_{L^2(\Omega)} 
\lesssim  h^{3/2}|\ln h|^{1/2} \|\cutoff e\|_{L^2(\Omega)}. 
$$
{\em Proof of \eqref{eq:lemma:weighted-w-wh-2}:}
Define the bidual problem as follows: 
\begin{subequations}
\label{eq:dual-dual-k>1}
\begin{eqnarray}
\label{eq:dual-dual-k>1-1}
\widetilde a(z,v) - b(v,\psi) &=& 
\langle \widetilde\delta_\Gamma^{-1+2\varepsilon} e_w,v\rangle_\Omega \qquad \forall v \in H^1(\Omega), \\
\label{eq:dual-dual-k>1-2}
b(z,\mu) + c(\psi,\mu) &=& 0 \qquad \forall \mu \in H^{-1/2}(\Gamma). 
\end{eqnarray}
\end{subequations}
Then, we may proceed completely analogously as in the proof of 
\eqref{eq:lemma:weighted-w-wh-1} above. With the {\sl a priori} estimate \eqref{eq:lemma:regularity-weighted-rhs-vorn-2}
of Lemma~\ref{lemma:regularity-weighted-rhs-vorn}, we obtain the bound 
$
\|z\|_{H^{3/2+\varepsilon}(\Omega)} + \|\psi\|_{H^{\varepsilon}_{pw}} \lesssim \|\widetilde\delta_\Gamma^{-1/2+\varepsilon} e_w\|_{L^2(\Omega)}.
$
Therefore, classical approximation estimates give
\begin{eqnarray*}
\inf_{(\IV z,\IM \psi) \in X_h} \Big[
\|z - \IV z\|_{H^1(\Omega)}  + \|\psi - \IM \psi\|_{V}\Big]
&\lesssim h^{1/2+\varepsilon} \|\widetilde\delta_\Gamma^{-1/2+\varepsilon} e_w\|_{L^2(\Omega)}. 
\end{eqnarray*}
We employ the {\it a~priori estimate}~\eqref{eq:lemma:B32-regularity-3} of Lemma~\ref{lemma:B32-regularity} and argue as above to see
$$
\| \widetilde\delta_\Gamma^{-1/2+\varepsilon} e_w\|_{L^2(\Omega)} 
\lesssim h^{1/2+\varepsilon} \Big[ \|w - w_h\|_{H^1(\Omega)} + \|\lambda - \lambda_h\|_{V}\Big] 
\leq C h^{3/2+\varepsilon} \|\cutoff e\|_{L^2(\Omega)}. 
$$
This concludes the proof.
\end{proof}
We now turn to estimating $\|\widetilde\delta_\Gamma^{1/2+\varepsilon} \nabla (w - w_h)\|_{L^2(S_D)}$: 

\begin{lemma}
\label{lemma:w-wh-L1}
Let Assumptions~\ref{assumption:shift-theorem} and~\ref{assumption:shift-theorem-bidual-problem}
be valid. 
\begin{enumerate}[(i)]
\item 
There are constants $C_1$, $C_2>0$ such that the following is true: 
\label{item:lemma:w-wh-L1-i}
\begin{align*}
 \mbox{ If $k = 1$, then } &\quad 
 \|\widetilde\delta_\Gamma^{1/2} \nabla( w - w_h)\|_{L^2(\Omega)} \leq C_1 |\ln h|^{1/2} h^{3/2}\|\cutoff e\|_{L^2(\Omega)}. \\
 \mbox{ If $k > 1$ and  $\varepsilon \in (0,s_0-1/2]$, then } &\quad 
 \|\widetilde\delta_\Gamma^{1/2+\varepsilon} \nabla (w - w_h)\|_{L^2(\Omega)} \leq C_2 h^{3/2+\varepsilon}
 \|\cutoff e\|_{L^2(\Omega)}.
\end{align*}
The constant $C_1$ depends on the same quantities as the 
constant $C_1$ of Lemma~\ref{lemma:weighted-w-wh} and $\|\opA\|_{C^{0,1}(\overline{\Omega})}$; the constant 
$C_2 > 0$ depends on the same quantities as the constant $C_2$ 
in Lemma~\ref{lemma:weighted-w-wh} and additionally on $\|\opA\|_{C^{1,1}(\overline{\Omega})}$.

\item 
\label{item:lemma:w-wh-L1-ii}
For any fixed $D^\prime > 0$, we have 
$$
\|\nabla (w - w_h)\|_{L^2(\Omega \setminus S_{D^\prime})} 
\leq C_3 
\|\cutoff e\|_{L^2(\Omega)}
\begin{cases} 
h^{3/2} & \mbox{ if $k = 1$} \\
h^{1+s_0} & \mbox{ if $k > 1$.} 
\end{cases}
$$

The constant $C_3 > 0$ depends on $D^\prime$ and on 
the same quantities as the constants $C_1$, $C_2$ in (\ref{item:lemma:w-wh-L1-i}) for the cases 
$ k =1 $ and $k > 1$, respectively. 

\end{enumerate}
\end{lemma}

\begin{proof} 
{\em Proof of \eqref{item:lemma:w-wh-L1-i}:} The norm 
$\|\widetilde\delta_\Gamma^{1/2+\varepsilon} \nabla (w - w_h)\|_{L^2(\Omega)}$ 
is decomposed as 
\begin{equation}
\label{eq:lemma:w-wh-L1-7} 
\|\widetilde\delta_\Gamma^{1/2+\varepsilon} \nabla (w - w_h)\|_{L^2(\Omega)} \leq 
\|\widetilde \delta_\Gamma^{1/2+\varepsilon} \nabla(w - w_h)\|_{L^2(S_{ch})} + 
\|\widetilde\delta_\Gamma^{1/2+\varepsilon} \nabla(w - w_h)\|_{L^2(\Omega\setminus S_{ch})}
\end{equation}
for some fixed $c > 0$ (determined below in dependence on 
$\Omega$ and the shape regularity of the triangulation ${\mathcal T}_\Omega$) 
and each of these two contributions is estimated separately. 
We start with the simpler, first one: 
\begin{align}
\|\widetilde \delta_\Gamma^{1/2+\varepsilon} \nabla (w - w_h)\|_{L^2(S_{ch})} 
&\leq  (ch+h)^{1/2+\varepsilon} \|\nabla (w - w_h)\|_{L^2(S_{ch})} 
\lesssim h^{1/2+\varepsilon} \|\nabla(w - w_h)\|_{L^2(\Omega)}
\label{eq:lemma:w-wh-L1-100}
 \stackrel{Lem.~\ref{lemma:a-priori-dual}}{\lesssim} 
 h^{3/2+\varepsilon} \|\cutoff e\|_{L^2(\Omega)}. 
\end{align}
The second term, $\|\widetilde\delta_\Gamma^{1/2+\varepsilon}\nabla (w - w_h)\|_{L^1(\Omega\setminus S_{ch})}$ 
requires tools from the local error
analysis in FEM. The Galerkin orthogonality 
$$
a(w - w_h,v) = 0 \qquad \forall v \in V_h \cap H^1_0(\Omega)
$$
allows us to use the techniques of the local error analysis of FEM as described in \cite[Sec.~{5.3}]{wahlbin95}. 
This leads to the following estimate for arbitrary balls $B_r \subset B_{r^\prime}$ with the same 
center (implicitly, $r^\prime > r + O(h)$) 
\begin{equation}
\label{eq:wahlbin-1}
\|\nabla( w - w_h)\|_{L^2(B_r)} \lesssim \|\nabla (w - \Ihk w)\|_{L^2(B_{r^\prime})} + 
\frac{1}{r^\prime - r}\|w - w_h\|_{L^2(B_{r^\prime})}, 
\end{equation}
where $\Ihk w$ is a local approximant such as the one of Lemma~\ref{lemma:approximation-properties-Vh-Mh}.
By a covering argument, these local estimates can be combined into a global estimate of the following form, 
where for sufficiently small $c_1 \in (0,1) $ (depending only on $\Omega$ and the shape regularity of 
${\mathcal T}_\Omega$)
\begin{equation}
\label{eq:wahlbin-10}
\|\widetilde \delta_\Gamma^{1/2+\varepsilon} \nabla (w - w_h)\|_{L^2(\Omega\setminus S_{c h})}
\lesssim \|\widetilde \delta_\Gamma^{1/2+\varepsilon} \nabla (w - \Ihk w)\|_{L^2(\Omega\setminus S_{c_1 c h})} 
+ \|\widetilde \delta_\Gamma^{-1/2+\varepsilon} (w - w_h)\|_{L^2(\Omega\setminus S_{c_1 c h})}. 
\end{equation}
An implicit assumption is that $c_1 c h > 2h$. 
We emphasize that $\varepsilon  = 0$ is admissible in \eqref{eq:wahlbin-10}. 

We consider the cases $k = 1$ and $k > 1$ separately. 

For $k > 1$, we assume $\varepsilon \in (0, s_0-1/2]$ 
in \eqref{eq:wahlbin-10}. Employing in \eqref{eq:wahlbin-10} the bound 
\eqref{eq:lemma:weighted-w-wh-2} for the term
$\|\widetilde\delta_\Gamma^{-1/2+\varepsilon} (w - w_h)\|_{L^2(\Omega\setminus S_{c_1 c h})}$
and the fact that $k > 1$ together with the approximation properties of $\Ihk$,
we get for suitable $c_2 \in (0,1)$ (again depending only on $\Omega$ and the shape regularity of 
${\mathcal T}_\Omega$)
\begin{align}
\nonumber 
\|\widetilde\delta_\Gamma^{1/2+\varepsilon}\nabla (w - w_h)\|_{L^2(\Omega\setminus S_{c h})} 
&\stackrel{\eqref{eq:lemma:weighted-w-wh-2}}{\lesssim}
\|\widetilde\delta_\Gamma^{1/2+\varepsilon}\nabla (w - \Ihk w)\|_{L^2(\Omega\setminus S_{c_1 c h})} 
+ h^{3/2+\varepsilon} \|\cutoff e\|_{L^2(\Omega)}  \\
\label{eq:lemma:w-wh-L1-9}
&\lesssim h^2 \|\delta_\Gamma^{1/2+\varepsilon} \nabla^3 w\|_{L^2(\Omega\setminus S_{c_2 c_1 c h})} 
+ h^{3/2+\varepsilon} \|\cutoff e\|_{L^2(\Omega)}.
\end{align}
In this estimate, we have implicitly assumed that $c > 0$ is sufficiently large so that  
$c_2 c_1 c h > 2 h$. 
Combining Lemma~\ref{lemma:5.4}, \eqref{item:lemma:5.4-2}, \eqref{item:lemma:5.4-3}, and the 
regularity assertion \eqref{eq:regularity-of-w-2} allows us to conclude with yet another 
constant $c_3 \in (0,1)$ (depending on $\Omega$ and the shape regularity of ${\mathcal T}_\Omega$)
\begin{eqnarray}
\nonumber 
\|\delta_\Gamma^{1/2+\varepsilon} \nabla^3 w\|_{L^2(\Omega\setminus S_{c_2 c_1 c h})} 
& \stackrel{\text{Lem.~\ref{lemma:5.4}~\eqref{item:lemma:5.4-2}}}{\lesssim}& 
\|\delta_\Gamma^{-1/2+\varepsilon} \nabla^2 w\|_{L^2(\Omega\setminus S_{c_3 c_2 c_1 c h})} 
+ \|\delta_\Gamma^{1/2+\varepsilon} \nabla w\|_{L^2(\Omega\setminus S_{c_3 c_2 c_1 c h})} 
\\
&{\lesssim}&
h^{-1/2+\varepsilon} \|\nabla^2 w\|_{L^2(\Omega\setminus S_{c_3 c_2 c_1 c h})}
+ \|\nabla w\|_{L^2(\Omega\setminus S_{c_3 c_2 c_1 c h})}
\\
 &\stackrel{\text{Lem.~\ref{lemma:5.4}~\eqref{item:lemma:5.4-3}}}
\lesssim& h^{-1+2 \varepsilon} \|w\|_{H^{3/2+\varepsilon}(\Omega)}  \\
\label{eq:wahlbin-100}
&\stackrel{\eqref{eq:regularity-of-w-2}}\lesssim& h^{-1/2+\varepsilon} \|\cutoff e\|_{L^2(\Omega)}.
\end{eqnarray}
Again, the implicit assumption on $c$ is that $c_3 c_2 c_1 c  > \widetilde c$ with 
$\widetilde c$ given by Lemma~\ref{lemma:5.4}.
The final condition on $c$ therefore is $c > \max\{\widetilde c/(c_1 c_2 c_3),2/(c_1 c_2)\}$. 
The above estimates, namely, the combination of 
\eqref{eq:lemma:w-wh-L1-7}, 
\eqref{eq:lemma:w-wh-L1-100}, 
\eqref{eq:lemma:w-wh-L1-9}, 
\eqref{eq:wahlbin-100}
shows
for $k > 1$ and $\varepsilon \in (0,s_0-1/2]$ that 
$\|\widetilde\delta_\Gamma^{1/2+\varepsilon} \nabla (w - w_h)\|_{L^2(\Omega)} 
\lesssim h^{3/2+\varepsilon} \|\cutoff e\|_{L^2(\Omega)}$, which is the claimed estimate. 

The case $k = 1$ corresponds to the limiting situation $\varepsilon = 0$ 
in \eqref{eq:wahlbin-10}.
The procedure is 
analogous to that for the case $k>1$ except that we use 
$\|\delta_\Gamma^{1/2} \nabla (w - \Ihk w)\|_{L^2(\Omega\setminus S_{c_1 c h})}
\lesssim h \|\delta_\Gamma^{1/2} \nabla^2 w\|_{L^2(\Omega\setminus S_{c_2 c_1 c h})}$ 
and then  Lemma~\ref{lemma:5.4}, \eqref{item:lemma:5.4-1} in conjunction with 
\eqref{eq:regularity-of-w-1}. 
The contribution $\|\delta_\Gamma^{-1/2} (w - w_h)\|_{L^2(\Omega\setminus S_{c_1 c h})}$
is controlled with the aid of \eqref{eq:lemma:weighted-w-wh-1} of Lemma~\ref{lemma:weighted-w-wh}. 

{\em Proof of \eqref{item:lemma:w-wh-L1-ii}:} 
This follows more easily from \eqref{eq:wahlbin-1}. Since $D^\prime > 0$ is fixed, 
\eqref{eq:wahlbin-1} leads by a covering argument to 
\begin{equation}
\label{eq:lemma:w-wh-L1-500}
\|\nabla (w - w_h)\|_{L^2(\Omega\setminus S_{D^\prime}) }
\lesssim \|\nabla( w - \Ihk w)\|_{L^2(\Omega\setminus S_{D^{\prime\prime}})} 
+ \|\delta_\Gamma^{-1/2} (w - w_h)\|_{L^2(\Omega\setminus S_{D^{\prime\prime}})}
\end{equation}
for some $D^{\prime\prime} > 0$. Standard approximation properties of nodal interpolation
(cf.~Lemma~\ref{lemma:approximation-properties-Vh-Mh}) gives 
\begin{equation}
\label{eq:lemma:w-wh-L1-1000}
\|\nabla (w - \Ihk w)\|_{L^2(\Omega\setminus S_{D^{\prime\prime}})}
\lesssim 
\begin{cases} h^1 \|w\|_{H^2(\Omega \setminus S_{D^{\prime\prime\prime}})} & \mbox{ if $k = 1$} \\
              h^2 \|w\|_{H^3(\Omega \setminus S_{D^{\prime\prime\prime}})} & \mbox{ if $k > 1$.} 
\end{cases} 
\end{equation}
Interior regularity (note: $-\nabla \cdot (\opA \nabla w) = 0$ on $\Omega\setminus S_h$)
allows us to estimate 
\begin{equation}
\label{eq:lemma:w-wh-L1-2000}
\|w\|_{H^2(\Omega \setminus S_{D^{\prime\prime\prime}})} \lesssim \|w\|_{H^1(\Omega)} 
\lesssim h^{1/2} \|\cutoff e\|_{L^2(\Omega)} 
\qquad \mbox{ and } \qquad 
\|w\|_{H^3(\Omega \setminus S_{D^{\prime\prime\prime}})} \lesssim \|w\|_{H^1(\Omega)} 
\lesssim h^{1/2} \|\cutoff e\|_{L^2(\Omega)},
\end{equation}
where we employed~\eqref{eq:regularity-of-w}.
The term $\|\delta_\Gamma^{-1/2} (w - w_h)\|_{L^2(\Omega\setminus S_{D^{\prime\prime}})}$ 
is bounded by $h^{1+s_0} \|\cutoff e\|_{L^2(\Omega)}$ by \eqref{eq:lemma:weighted-w-wh-2} of 
Lemma~\ref{lemma:weighted-w-wh}. Inserting this and (\ref{eq:lemma:w-wh-L1-1000}) in
(\ref{eq:lemma:w-wh-L1-500}) yields the result. 
\end{proof}

\subsection{Proof of estimate~\eqref{eq:thm2} of Theorem~\ref{thm:approximation-of-varphi}.}
\label{sec:proof-of-thm:error-on-strip}

\begin{numberedproof}{Theorem~\protect{\ref{thm:approximation-of-varphi}, estimate~(\ref{eq:thm2})}} 
Theorem~\ref{thm:approximation-of-varphi} disregards variational crimes. 
Therefore, Lemma~\ref{lemma:error-on-strip} applies with 
$\varepsilon_1 \equiv 0$ and $\varepsilon_2 \equiv 0$. 
Together with \eqref{eq:estimate-most-terms-of-lemma:a-priori-dual}--\eqref{eq:estimate-part-of-tilde-a-of-lemma:a-priori-dual},
Lemma~\ref{lemma:error-on-strip} yields for 
arbitrary $\IV u \in V_h$, $\IM \varphi \in M_h$: 
\begin{eqnarray}
\label{eq:proof-of-thm:error-on-strip-100}
\|\cutoff e\|^2_{L^2(\Omega)} 
\lesssim 
h \|\cutoff e \|_{L^2(\Omega)} 
\Big[ 
\|u - \IV u\|_{H^{1/2}(\Gamma)} + \|\varphi - \IM \varphi\|_{H^{-1/2}(\Gamma)}
\Big] + 
\left| a(u - \IV u,w - w_h)\right|. 
\end{eqnarray}
Lemma~\ref{lemma:w-wh-L1}, (i) provides  
\begin{align*}
\|\widetilde\delta_\Gamma^{1/2+\varepsilon} \nabla(w - w_h)\|_{L^2(S_D)} 
& \lesssim h^{3/2} \|\cutoff e\|_{L^2(\Omega)}
\begin{cases}
|\ln h|^{1/2}  & \mbox{ if $\varepsilon = 0$} \\
h^\varepsilon & \mbox{ if $\varepsilon \in (0,s_0-1/2]$ and $k > 1$}
\end{cases}
\end{align*}
With~\eqref{eq:cor:randnaehe-10}, we hence get for the case $k > 1$
together with Lemma~\ref{lemma:w-wh-L1}, \eqref{item:lemma:w-wh-L1-ii}
\begin{eqnarray}
\nonumber 
\left| a(u - \IV u,w - w_h)\right| 
&\stackrel{\eqref{eq:cor:randnaehe-10}}\lesssim& 
\|\widetilde \delta_\Gamma^{-1/2-\varepsilon} \nabla (u - \IV u)\|_{L^2(S_D)} 
\|\widetilde \delta_\Gamma^{1/2+\varepsilon} \nabla (w - w_h)\|_{L^2(S_D)}  
+ 
\|\nabla (u - \IV u) \|_{L^2(\Omega\setminus S_D)} 
\|\nabla (w - w_h) \|_{L^2(\Omega\setminus S_D)}  \\
& \stackrel{\text{Lem.~\ref{lemma:w-wh-L1}}}\lesssim &
\Big[
\widetilde \delta_\Gamma^{-1/2-\varepsilon} \nabla (u - \IV u)\|_{L^2(S_D)} \,
h^{3/2+\varepsilon}
+ \|\nabla (u - \IV u) \|_{L^2(\Omega\setminus S_D)} \,
h^{1+s_0}\Big]\,\|\cutoff e\|_{L^2(\Omega)}.
\label{eq:proof-of-thm:error-on-strip-199}
\end{eqnarray}
We choose $\IV u = \Ihk u$.
{}From Lemma~\ref{lemma:approximation-properties-Vh-Mh}, we get 
\begin{align*}
\|\widetilde\delta_\Gamma^{-1/2-\varepsilon} \nabla (u - \Ihk u)\|_{L^2(S_D)} 
&\lesssim h^k \|\nabla^{k+1} u\|_{B^{1/2}_{2,1}(\Omega)} 
\begin{cases} 
|\ln h|^{1/2} & \mbox{ if $\varepsilon = 0$}, \\
h^{-\varepsilon} & \mbox{ if $\varepsilon > 0$}.  
\end{cases}
\end{align*}
With this and standard approximation properties of $\Ihk$, we obtain
\begin{eqnarray}
\left| a(u - \Ihk u,w - w_h)\right| 
\nonumber 
&\stackrel{\text{Lem.~\ref{lemma:approximation-properties-Vh-Mh}}}\lesssim &
\Big[ h^{k-\varepsilon} \|\nabla^{k+1} u\|_{B^{1/2}_{2,1}(\Omega)} h^{3/2+\varepsilon} + 
h^k \|\nabla^{k+1} u\|_{L^2(\Omega)} h^{1+s_0}
\Big] \|\cutoff e\|_{L^2(\Omega)} \\
\label{eq:proof-of-thm:error-on-strip-200}
&\lesssim& h^{k+3/2} \|u\|_{B^{k+3/2}_{2,1}(\Omega)} \|\cutoff e\|_{L^2(\Omega)}, 
\end{eqnarray}
where, in the last step, we employed the continuous embedding $B^{k+3/2}_{2,1}(\Omega) \subset H^{k+1}(\Omega)$
and the assumption $s_0 > 1/2$. 

For the lowest order case $k = 1$, the estimate corresponding to 
(\ref{eq:proof-of-thm:error-on-strip-200}) is 
\begin{equation}
\label{eq:proof-of-thm:error-on-strip-250}
\left| a(u - \Ihk u,w - w_h)\right| \lesssim h^{k+3/2} |\ln h| \|u\|_{B^{k+3/2}_{2,1}(\Omega)} \|\cutoff e\|_{L^2(\Omega)}.
\end{equation}
Therefore, the cases $k=1$ and $k > 1$ can be combined into the bound 
\begin{equation}
\label{eq:proof-of-thm:error-on-strip-252}
\left| a(u - \Ihk u,w - w_h)\right| \lesssim 
h^{k+3/2} (1 + \delta_{k,1} |\ln h|) \|u\|_{B^{k+3/2}_{2,1}(\Omega)} \|\cutoff e\|_{L^2(\Omega)}.
\end{equation}
The estimate (\ref{eq:foo-4}) yields 
$\|u - \Ihk u\|_{H^{1/2}(\Gamma)} \lesssim h^{k+1/2} \|u\|_{B^{k+3/2}_{2,1}(\Omega)}$. 
A suitable choice of $\IM \varphi$ provides $\|\varphi - \IM \varphi\|_{H^{-1/2}(\Gamma)} 
\lesssim h^{k+3/2} \|\varphi\|_{H^k_{pw}(\Gamma)}$. Inserting these two bounds 
and (\ref{eq:proof-of-thm:error-on-strip-252}) in 
\eqref{eq:proof-of-thm:error-on-strip-100} finally produces 
\begin{eqnarray}
\label{eq:proof-of-thm:error-on-strip-10}
\|\cutoff e\|_{L^2(\Omega)} 
\lesssim h^{k+3/2} 
 ( 1 + \delta_{k,1} |\ln h|)\|u\|_{B^{k+3/2}_{2,1}(\Omega)} + 
h^{k+3/2} \|\varphi\|_{H^{k}_{pw}(\Gamma)}. 
\end{eqnarray}
This concludes the proof.
\end{numberedproof}

\subsection{Extensions.}
\label{sec:randnaehe}

Estimate~\eqref{eq:proof-of-thm:error-on-strip-199} in
the proof of Theorem~\ref{thm:approximation-of-varphi} shows that we have actually
obtained the following result: 

\begin{corollary}[best approximation property]
\label{cor:best-approximation}
Assume the hypotheses of Theorem~\ref{thm:approximation-of-varphi}. 
Then, for arbitrary $\IV u \in V_h$ and $\IM \varphi \in M_h$ there holds: 
\begin{enumerate}[(i)]
\item 
\label{item:cor:best-approximation-i}
If $ k = 1$ then  
\begin{eqnarray*}
\|u - u_h\|_{L^2(S_h)} &\lesssim & 
h^{3/2} |\ln h|^{1/2} \|\widetilde \delta_\Gamma^{-1/2} \nabla (u - \IV u)\|_{L^2(\Omega)} 
+ h \|u - \IV u\|_{H^{1/2}(\Gamma)} + h\|\varphi - \IM \varphi\|_{H^{-1/2}(\Gamma)}. 
\end{eqnarray*}
\item 
\label{item:cor:best-approximation-ii}
If $k > 1$ then for arbitrary $\varepsilon \in (0,s_0-1/2]$ and $D > 0$
\begin{eqnarray*}
\|u - u_h\|_{L^2(S_h)} &\lesssim & 
h^{3/2+\varepsilon} \|\widetilde \delta_\Gamma^{-1/2-\varepsilon} \nabla (u - \IV u)\|_{L^2(S_D)} 
+ h^{1+s_0} \|\nabla(u - \IV u)\|_{L^2(\Omega)} 
+ h \|u - \IV u\|_{H^{1/2}(\Gamma)} + h\|\varphi - \IM \varphi\|_{H^{-1/2}(\Gamma)}.
\end{eqnarray*}
\end{enumerate}
The constant $C$ depends on the same quantities as in Theorem~\ref{thm:approximation-of-varphi} and 
additionally on $D$ and $\varepsilon$ in 
(\ref{item:cor:best-approximation-ii}). \hfill{\small$\square$}
\end{corollary}

The statement of 
Corollary~\ref{cor:best-approximation} (\ref{item:cor:best-approximation-ii}) for the case $k > 1$ suggests that 
the $B^{k+3/2}_{2,1}$-regularity of the solution is only required near the 
coupling boundary $\Gamma$, while away from $\Gamma$ a weaker estimate is sufficient. 
The following result is meant to illustrate this point; it 
does not lay claim on sharpness of the regularity requirements
(in fact, the presence of the small factor $h^{s_0-1/2}$ is a clear indication of 
a lack of sharpness):

\begin{corollary}[reduced regularity away from $\Gamma$]
\label{cor:randnaehe}
Assume the hypotheses of Theorem~\ref{thm:approximation-of-varphi}. Let $k \ge 2$. 
Let $u \in B^{k+3/2}_{2,1}(S_D) \cap H^{k+1}(\Omega)$ for some fixed $D> 0$. 
Let $\varphi \in H^k_{pw}(\Gamma)$. Then: 
\begin{eqnarray}
\label{eq:cor:randnaehe-1}
\|u - u_h\|_{L^2(S_h)} &\leq& C h^{3/2+k} \Big[ \|u\|_{B^{k+3/2}_{2,1}(S_D)} + 
h^{s_0-1/2} \|u\|_{H^{k+1}(\Omega)} + \|\varphi\|_{H^{k}_{pw}(\Gamma)}\Big], \\
\label{eq:cor:randnaehe-2}
\|\varphi - \varphi_h\|_{V} &\leq& 
C h^{1/2+k} \Big[ \|u\|_{B^{k+3/2}_{2,1}(S_D)} + 
h^{s_0-1/2} \|u\|_{H^{k+1}(\Omega)} + \|\varphi\|_{H^{k}_{pw}(\Gamma)}\Big].
\end{eqnarray}
The constant $C$ depends on the same quantities as in Theorem~\ref{thm:approximation-of-varphi} and 
additionally on $D$. 
\end{corollary}

\begin{proof}
Follows from Corollary~\ref{cor:best-approximation} (\ref{item:cor:best-approximation-ii}), 
the assumption $k \ge 2$, the special choice $\IV u = \Ihk u$, and Lemma~\ref{lemma:weighted-embedding}. 
\end{proof}

\subsection{Variational crimes}
\label{sec:variational-crimes}
We recall that our proof of Theorem~\ref{thm:approximation-of-varphi} 
does not assess the impact of variational crimes, i.e., 
it assumes $\varepsilon_1 \equiv 0$ and $\varepsilon_2 \equiv 0$
in the Galerkin orthogonalities \eqref{eq:orthogonalities-primal}. As mentioned above, 
the terms $\varepsilon_1$, $\varepsilon_2$ were introduced 
in Lemma~\ref{lemma:error-on-strip} in order to be able to account for certain types of variational crimes, which 
is the topic of the present section.
We start with a standard result that incorporates the effect of approximating the data $u_0$ and $\phi_0$ 
in the approximation result of Theorem~\ref{thm:approximation-of-varphi} (\ref{item:thm:approximation-of-varphi-i}):
\begin{lemma}
\label{lemma:standard-variational-crimes}
Let $u_0 \in H^{k+1/2}_{pw}(\Gamma) \cap H^1(\Gamma)$ and $\phi_0 \in H^{k-1/2}_{pw}(\Gamma)$.
Let $(u,\varphi) \in X$ be the solution of \eqref{eq:rhs-weak}--\eqref{eq:weak-form}.
Let $(u_h,\varphi_h)\in X_h$ be the Galerkin solution of~\eqref{eq:weak-form-FEM}
with $u_0$ and $\phi_0$ in \eqref{eq:rhs-weak} replaced by approximations $\Pi^k_h u_0$ 
and $\Pi^{k-1}_h \phi_0$ that have the following approximation properties: 
\begin{align}
\label{eq:lemma:standard-variational-crimes-100}
\|u_0 - \Pi^k_h u_0\|_{H^{1/2}(\Gamma)} \leq C_{\text{apx}} h^{k} \|u_0\|_{H^{k+1/2}_{pw}(\Gamma)}  
&\qquad \mbox{ and } \qquad 
\|\phi_0 - \Pi^{k-1}_h \phi_0\|_{H^{-1/2}(\Gamma)}  \leq C_{\text{apx}} h^{k} \|\phi_0\|_{H^{k-1/2}_{pw}(\Gamma)}. 
\end{align}
Under the assumptions of Theorem~\ref{thm:approximation-of-varphi} (\ref{item:thm:approximation-of-varphi-i}) 
there holds 
\begin{equation*}
\|u - u_h\|_{H^1(\Omega)} + \|\varphi-  \varphi_h\|_{H^{-1/2}(\Gamma)}  
\leq C h^k \left[ \|u\|_{H^{k+1}(\Omega)} + \|\varphi\|_{H^{k-1/2}_{pw}(\Gamma)} 
+ \|u_0\|_{H^{k+1/2}_{pw}(\Gamma)} + \|\phi_0\|_{H^{k-1/2}_{pw}(\Gamma)} \right]. 
\end{equation*}
The constant depends on the constant $C$ of Theorem~\ref{thm:approximation-of-varphi} (\ref{item:thm:approximation-of-varphi-i})
as well as $C_{\text{apx}}$ of (\ref{eq:lemma:standard-variational-crimes-100}).
\end{lemma}
\begin{proof}
The Galerkin orthogonality~\eqref{eq:orthogonalities-primal} for the primal problem now 
includes the linear functionals
\begin{subequations}
\label{eq:specific-errors} 
\begin{align} 
 \varepsilon_1(v) &:= \langle \phi_0 - \Pi^{k-1}_h \phi_0, v\rangle + \langle \hyp (u_0 - \Pi^k_h u_0),v\rangle,\\
 \varepsilon_2(\psi) &:= \langle (\frac{1}{2} - \dlo) (u_0 - \Pi^k_h u_0),\psi\rangle,
\end{align}
\end{subequations}
on the right-hand side. Hence, the standard convergence theory for Galerkin methods with variational crimes 
implies that the error bound (\ref{eq:standard-estimate-2}) is augmented by terms involving $\varepsilon_1$ and $\varepsilon_2$: 
\begin{equation}
\label{eq:standard-estimate-3} 
\|u - u_h\|_{H^1(\Omega)} + \|\varphi - \varphi_h\|_{H^{-1/2}(\Gamma)} 
\lesssim 
\inf_{(v,\psi) \in X_h} \left( \|u - v\|_{H^1(\Omega)} + \|\varphi - \psi\|_{H^{-1/2}(\Gamma)}\right) 
+  
\sup_{(v,\psi) \in X} \left( \frac{|\varepsilon_1(v)|}{\|v\|_{H^1(\Omega)}} + 
\frac{|\varepsilon_2(\psi)|}{\|\psi\|_{H^{-1/2}(\Gamma)} }\right).
\end{equation}
The infimum in (\ref{eq:standard-estimate-3}) has been treated in (\ref{eq:standard-estimate-2}). The supremum
in (\ref{eq:standard-estimate-3}) can be bounded in the desired fashion with the approximation assumptions 
(\ref{eq:lemma:standard-variational-crimes-100}) and the mapping properties the operators $\hyp$ and $1/2-\dlo$.
\end{proof}
In Theorem~\ref{thm:variational-crimes} below, we show that (up to logarithmic terms) also the 
improved convergence rates
for $\|\varphi - \varphi_h\|_{H^{-1/2}(\Gamma)}$ shown in Theorem~\ref{thm:approximation-of-varphi} (\ref{item:thm:approximation-of-varphi-ii}) 
are retained if the data $u_0$, $\phi_0$ are replaced with approximations~\eqref{eq:lemma:standard-variational-crimes-100}. The proof of this result requires us to 
address a technical issue, namely, the slight mismatch between
the regularity of the dual solution $w$ available to us and the regularity needed for the trace operator to be 
well-defined: we have $w \in B^{3/2}_{2,\infty}(\Omega)$ but for $\nabla w$ to have an $L^2(\Gamma)$-trace on 
$\Gamma$, we need $w \in B^{3/2}_{2,1}(\Omega)$, which is slightly stronger. The following 
lemma shows that for $w \in B^{n+3/2}_{2,\infty}(\Omega)$, 
one can construct a nearby function in the space $B^{n+3/2}_{2,1}(\Omega)$ that does have a trace. 

\begin{lemma} 
\label{lemma:B12infty-trace}
Let $\Omega \subset \BbbR^d$ be a bounded Lipschitz domain with boundary $\Gamma:= \partial\Omega$ and $n \in \BbbN_0$. 
\begin{enumerate}[(i)] 
\item 
\label{item:lemma:B12infty-trace-i}
There is $C > 0$ 
such that for every $w \in B^{n+1/2}_{2,\infty}(\Omega)$ and every $\varepsilon \in (0,1]$, one can find 
$w_\varepsilon  \in H^{n+1}(\Omega)$ such that 
with $\theta = \frac{n+1/2}{n+1}$
\begin{equation}
\label{eq:lemma:B12infty-trace-10}
\|w - w_\varepsilon\|_{L^2(\Omega)} + 
\varepsilon^{n/(n+1)} \|w - w_\varepsilon\|_{H^n(\Omega)} + 
\varepsilon \|w_\varepsilon\|_{H^{n+1}(\Omega)} \leq 
C \varepsilon^{\theta} \|w\|_{B^{n+1/2}_{2,\infty}(\Omega)}, 
\qquad \|w_\varepsilon\|_{B^{n+1/2}_{2,1}(\Omega)} \leq C (1+|\ln \varepsilon |) \|w\|_{B^{n+1/2}_{2,\infty}(\Omega)}. 
\end{equation}
\item 
\label{item:lemma:B12infty-trace-ii}
Fix $\delta \in (0,1]$. There is $C > 0$ such that for every $\lambda \in B^0_{2,\infty}(\Gamma)$ 
and every $\varepsilon \in (0,1]$, one can find $\lambda_\varepsilon \in H^{\delta}(\Gamma)$ such that 
$$
\|\lambda - \lambda_\varepsilon\|_{H^{-\delta}(\Gamma)} 
+ \varepsilon^{1/2} \|\lambda - \lambda_\varepsilon\|_{B^0_{2,\infty}(\Gamma)}
+ \varepsilon \|\lambda_\varepsilon\|_{H^\delta (\Gamma)} 
\leq C \varepsilon^{1/2} \|\lambda\|_{B^0_{2,\infty}(\Gamma)}, 
\qquad 
\|\lambda_\varepsilon\|_{L^2(\Gamma)} \leq C \sqrt{1+|\ln \varepsilon|} \|\lambda\|_{B^0_{2,\infty}(\Gamma)}.
$$
\end{enumerate}
The constant $C>0$ depends only on $\Omega$, $n$, and $\delta$.
\end{lemma}

\begin{proof} 
{\em Proof of \eqref{item:lemma:B12infty-trace-i}:}
First, we recall that the reiteration theorem~\cite[Thm.~26.3]{tartar07} allows us to define the Besov spaces 
$B^s_{2,q}(\Omega)$ by interpolation between Sobolev spaces $H^{s_1}(\Omega)$ and $H^{s_2}(\Omega)$. We take specifically 
$s_1 = 0$ and $s_2 = n+1$. Then $B^{n+1/2}_{2,q}(\Omega) = (L^2(\Omega),H^{n+1}(\Omega))_{\theta,q}$ with $\theta = (n+1/2)/(n+1)$. 
We recall the definition of the pertinent $K$-functional
$K(t,u) = \inf_{v \in H^{n+1}(\Omega)} \|u - v\|_{L^2(\Omega)} + t \|v\|_{H^{n+1}(\Omega)}$ 
and the corresponding interpolation norm from~\eqref{eq:intnorm}.
By definition of the interpolation space ${B^{n+1/2}_{2,\infty}(\Omega)}$, one can find, for every 
$\varepsilon > 0$, a function $w_\varepsilon \in H^{n+1}(\Omega)$ such that 
\begin{equation}
\label{eq:lemma:B12infty-trace-1}
\|w - w_\varepsilon\|_{L^2(\Omega)} + \varepsilon \|w_\varepsilon\|_{H^{n+1}(\Omega)} 
\leq C \varepsilon^{\theta} \|w\|_{B^{n+1/2}_{2,\infty}(\Omega)} 
\qquad \mbox{ and } \qquad 
\|w - w_\varepsilon\|_{B^{n+1/2}_{2,\infty}(\Omega)} \leq C \|w\|_{B^{n+1/2}_{2,\infty}(\Omega)}. 
\end{equation}
The first of these estimates follows from the definition, the second one is shown in \cite[Lemma]{bramble-scott78}.  
In order to complete the 
first bound in \eqref{eq:lemma:B12infty-trace-10}, we recall that the reiteration theorem~\cite[Thm.~26.3]{tartar07} yields
$H^n(\Omega) = (L^2(\Omega),B_{2,\infty}^{n+1/2}(\Omega))_{\mu,2}$ with $\mu = n / (n+1/2)$.
The interpolation inequality~\eqref{eq:intest} yields
$$
\|w - w_\varepsilon\|_{H^n(\Omega)} 
\lesssim 
\|w - w_\varepsilon\|_{L^2(\Omega)}^{1-n/(n+1/2)} 
\|w - w_\varepsilon\|_{B^{n+1/2}_{2,\infty}(\Omega)}^{n/(n+1/2)}
\lesssim \varepsilon^{\theta (1-n/(n+1/2))} \|w\|_{B^{n+1/2}_{2,\infty}(\Omega)} 
= \varepsilon^{-n/(n+1)}\varepsilon^\theta \|w\|_{B^{n+1/2}_{2,\infty}(\Omega)} .
$$
We turn to the second bound in \eqref{eq:lemma:B12infty-trace-10}. We first mention 
that, as it is shown in \cite[Chap.~6, Sec.~7]{devore93}, we may replace the integral 
over $(0,\infty)$ in (\ref{eq:intnorm}) by an integral over $(0,1)$. Next, 
we have the simple triangle inequality $K(t,w_\varepsilon) \leq K(t,w) 
+ \|w - w_\varepsilon\|_{L^2(\Omega)}$.  A second bound for $K(t,w_\varepsilon)$ is obtained
by taking $v = 0$ in the defining infimum: $K(t,w_\varepsilon) \leq t \|w_\varepsilon\|_{H^{n+1}(\Omega)}$. 
Put together, we arrive at 
$K(t,w_\varepsilon) \leq \min\{t \|w_\varepsilon\|_{H^{n+1}(\Omega)}, K(t,w) + \|w - w_\varepsilon\|_{L^2(\Omega)}\}$. 
We compute for $w_\varepsilon \in H^{n+1}(\Omega)$
\begin{align*}
\|w_\varepsilon\|_{B^{n+1/2}_{2,1}(\Omega)} & \simeq
\int_{t=0}^\varepsilon t^{-\theta} K(t,w_\varepsilon)\frac{dt}{t} + \int_{t=\varepsilon}^1 t^{-\theta} 
K(t,w_\varepsilon) \frac{dt}{t}
\le \varepsilon^{1-\theta} \|w_\varepsilon\|_{H^{n+1}(\Omega)} + \varepsilon^{-\theta} \|w - w_\varepsilon\|_{L^2(\Omega)} 
+\int_{t = \varepsilon}^1 t^{-\theta} K(t,w) \frac{dt}{t}  \\
&\stackrel{\eqref{eq:lemma:B12infty-trace-1}}{\lesssim }
\varepsilon^{1-\theta} \varepsilon^{\theta-1}\|w_\varepsilon\|_{B^{n+1/2}_{2,\infty}(\Omega)} 
+ \varepsilon^{\theta} \varepsilon^{-\theta} \|w\|_{B^{n+1/2}_{2,\infty} (\Omega)} 
+ (1+|\ln \varepsilon|) \|w\|_{B^{n+1/2}_{2,\infty}(\Omega)} 
\lesssim (1+|\ln \varepsilon|) \|w\|_{B^{n+1/2}_{2,\infty}(\Omega)},
\end{align*}
where we finally used $\|w_\varepsilon\|_{B^{n+1/2}_{2,\infty}(\Omega)}
\leq \|w \|_{B^{n+1/2}_{2,\infty}(\Omega)} + 
\|w - w_\varepsilon\|_{B^{n+1/2}_{2,\infty}(\Omega)}
\lesssim \|w\|_{B^{n+1/2}_{2,\infty}(\Omega)}$ from~\eqref{eq:lemma:B12infty-trace-1}.

{\em Proof of \eqref{item:lemma:B12infty-trace-ii}:} 
We proceed by similar arguments as in the proof of \eqref{item:lemma:B12infty-trace-i}:
We exploit the characterizations $L^2(\Gamma) = (H^{-\delta}(\Gamma),H^\delta(\Gamma))_{1/2,2}$ 
and $B^0_{2,\infty}(\Gamma) = (H^{-\delta}(\Gamma),H^\delta(\Gamma))_{1/2,\infty}$. 
{}From the properties of the $K$-functional, we get 
the existence of $\lambda_\varepsilon$ such that 
$$
\|\lambda - \lambda_\varepsilon\|_{H^{-\delta}(\Gamma)} + \varepsilon \|\lambda_\varepsilon\|_{H^\delta(\Gamma)} 
\leq C \varepsilon^{1/2} \|\lambda\|_{B^0_{2,\infty}(\Gamma)}, 
\qquad 
\|\lambda - \lambda_\varepsilon\|_{B^0_{2,\infty}(\Gamma)} \leq C 
\|\lambda \|_{B^0_{2,\infty}(\Gamma)},  
$$
where the second bound follows again from \cite{bramble-scott78}. It remains to bound 
$\|\lambda_\varepsilon\|_{L^2(\Gamma)}$. As above, we observe 
$K(t,\lambda_\varepsilon) \leq \min\{t \|\lambda_\varepsilon\|_{H^\delta(\Gamma)}, K(t,\lambda) + \|\lambda - \lambda_\varepsilon\|_{H^{-\delta}(\Gamma)}\}$. Hence, 
\begin{align*}
\|\lambda_\varepsilon\|^2_{L^2(\Gamma)} &\simeq
\int_{t=0}^1 \left| t^{-1/2} K(t,\lambda_\varepsilon)\right|^2 \frac{dt}{t}
\lesssim
\int_{t=0}^{\varepsilon} \|\lambda_\varepsilon\|^2_{H^\delta(\Gamma)}\,dt + 
\int_{t=\varepsilon}^1 \left| t^{-1/2} K(t,\lambda)\right|^2\frac{dt}{t}+ 
\int_{t=\varepsilon}^1 t^{-2} \|\lambda - \lambda_\varepsilon\|^2_{H^{-\delta}(\Gamma)}\,dt \\
&\le \varepsilon \|\lambda_\varepsilon\|^2_{H^{\delta}(\Gamma)} 
+ |\ln \varepsilon| \sup_{t > 0} |t^{-1/2} K(t,\lambda)|^2 + \varepsilon^{-1} \|\lambda - \lambda_\varepsilon\|^2_{H^{-\delta}(\Gamma)} 
\lesssim (1 + |\ln \varepsilon|) \|\lambda\|^2_{B^0_{2,\infty}(\Gamma)},  
\end{align*}
from which the result follows. 
\end{proof}

\begin{theorem}[variational crimes]
\label{thm:variational-crimes}
Let $u_0 \in H^{k+1}_{pw}(\Gamma) \cap H^1(\Gamma)$ and $\phi_0 \in H^{k}_{pw}(\Gamma)$.
Let $(u,\varphi) \in X$ be the solution of \eqref{eq:rhs-weak}--\eqref{eq:weak-form}.
Let $(u_h,\varphi_h)\in X_h$ be the Galerkin solution of~\eqref{eq:weak-form-FEM}
with $u_0$ and $\phi_0$ in \eqref{eq:rhs-weak} replaced by approximations $\Pi^k_h u_0$ 
and $\Pi^{k-1}_h \phi_0$ that have the following approximation properties: 
\begin{align}
\label{eq:cor:variational-crimes-100}
h^{-1/2} \|u_0 - \Pi^k_h u_0\|_{L^2(\Gamma)} + 
\|u_0 - \Pi^k_h u_0\|_{H^{1/2}(\Gamma)} &\leq C_{\text{apx}} h^{k+1/2} \|u_0\|_{H^{k+1}_{pw}(\Gamma)},  \\
\label{eq:cor:variational-crimes-200}
h^{-1/2} \|\phi_0 - \Pi^{k-1}_h \phi_0\|_{H^{-1}(\Gamma)} + 
\|\phi_0 - \Pi^{k-1}_h \phi_0\|_{H^{-1/2}(\Gamma)}  &\leq C_{\text{apx}} h^{k+1/2} \|\phi_0\|_{H^k_{pw}(\Gamma)}. 
\end{align}
Under the assumptions of Theorem~\ref{thm:approximation-of-varphi} (\ref{item:thm:approximation-of-varphi-ii})
it holds
\begin{eqnarray}
\label{eq:cor:variational-crimes-1029}
\|\varphi - \varphi_h\|_{H^{-1/2}(\Gamma)} &\leq& C h^{k+1/2} |\ln h| 
\Big[ \|u\|_{B^{k+3/2}_{2,1}(\Omega)} + \|u_0\|_{H^{k+1}_{pw}(\Gamma)} + \|\varphi\|_{H^k_{pw}(\Gamma)}
+ \|\phi_0\|_{H^k_{pw}(\Gamma)}\Big],\\
\label{eq:cor:variational-crimes-1030}
\|u - u_h\|_{L^2(S_h)} &\leq& C h^{k+3/2} |\ln h| 
\Big[ \|u\|_{B^{k+3/2}_{2,1}(\Omega)} + \|u_0\|_{H^{k+1}_{pw}(\Gamma)} + \|\varphi\|_{H^k_{pw}(\Gamma)}
+ \|\phi_0\|_{H^k_{pw}(\Gamma)}\Big].
\end{eqnarray}
The constant $C >0$ in 
\eqref{eq:cor:variational-crimes-1029}--\eqref{eq:cor:variational-crimes-1030}
depends on the same quantities as the constant $C$ in 
Theorem~\ref{thm:approximation-of-varphi} (\ref{item:thm:approximation-of-varphi-ii}) 
and additionally on $C_{\text{apx}}$.
\end{theorem}%

\begin{proof}
As in the proof of Lemma~\ref{lemma:standard-variational-crimes}
The Galerkin orthogonality~\eqref{eq:orthogonalities-primal} for the primal problem now 
includes the linear functionals
\begin{align*} 
 \varepsilon_1(v) &:= \langle \phi_0 - \Pi^{k-1}_h \phi_0, v\rangle + \langle \hyp (u_0 - \Pi^k_h u_0),v\rangle,\\
 \varepsilon_2(\psi) &:= \langle (\frac{1}{2} - \dlo) (u_0 - \Pi^k_h u_0),\psi\rangle,
\end{align*}
on the right-hand side. With Lemma~\ref{lemma:error-on-strip}, we observe that~\eqref{eq:proof-of-thm:error-on-strip-100} then holds with additionally the summand $\varepsilon_1(w_h)-\varepsilon_2(\lambda_h)$ on the
right-hand side. Recall that $e=u-u_h$ is the FEM-part of the primal error, and $(w,\lambda)\in X$ is
the solution of the dual problem~\eqref{eq:dual-problem-concrete} with Galerkin approximation
$(w_h,\lambda_h)\in X_h$.
Arguing along the lines of the proof of Theorem~\ref{thm:approximation-of-varphi}, estimate~\eqref{eq:thm2}, 
we see that additional error terms 
$$
\frac{|\varepsilon_1(w_h)| + |\varepsilon_2(\lambda_h)|}{\|\cutoff e\|_{L^2(\Omega)}} 
$$
arise. We now claim the following two bounds: 
\begin{align}
\label{eq:crimes1}
 |\varepsilon_1(w_h)| &\lesssim h^{k+3/2} |\ln h| \Big[ \|u_0\|_{H^{k+1}_{pw}(\Gamma)} 
+ \|\phi_0\|_{H^{k}_{pw}(\Gamma)}\Big]  \|\cutoff e\|_{L^2(\Omega)},\\
\label{eq:crimes2}
 |\varepsilon_2(\lambda_h)| &\lesssim h^{k+3/2} |\ln h|^{1/2} \|u_0\|_{H^{k+1}_{pw}(\Gamma)} \|\cutoff e\|_{L^2(\Omega)}.
\end{align}
{}From the {\sl a~priori} estimate~\eqref{eq:lemma:B32-regularity-2} of Lemma~\ref{lemma:B32-regularity},
we obtain $\|w\|_{B^{3/2}_{2,\infty}(\Omega)} \lesssim h^{1/2} \|\cutoff e\|_{L^2(\Omega)}$.
Lemma~\ref{lemma:B12infty-trace} with $\varepsilon = h^2$ and $n=1$ 
provides an approximation $w_\varepsilon \in H^2(\Omega)$ to $w$.
(We still write 
$w_\varepsilon$ to avoid confusion with the Galerkin approximation $w_h$.)
The estimates~\eqref{eq:lemma:B12infty-trace-10} take the form
\begin{subequations}\label{eq:crimes3}
\begin{align}\label{eq:crimes3a}
&\|w - w_\varepsilon\|_{L^2(\Omega)} + 
h\, \|w - w_\varepsilon\|_{H^1(\Omega)} + 
h^2 \|w_\varepsilon\|_{H^{2}(\Omega)} 
\lesssim h^{3/2} \|w\|_{B^{3/2}_{2,\infty}(\Omega)}
\lesssim h^2 \|\cutoff e\|_{L^2(\Omega)},\\
\label{eq:crimes3b}
&\|w_\varepsilon\|_{B^{3/2}_{2,1}(\Omega)} 
\lesssim |\ln h| \, \|w\|_{B^{3/2}_{2,\infty}(\Omega)}
 \lesssim h^{1/2}\, |\ln h| \, \|\cutoff e\|_{L^2(\Omega)}.
\end{align}
\end{subequations}
Linearity of $\varepsilon_1$ yields
\begin{align}
\label{eq:crimes-10}
 |\varepsilon_1(w_h)|
 \le |\varepsilon_1(w-w_h)| + |\varepsilon_1(w-w_\varepsilon)|  + |\varepsilon_1(w_\varepsilon)|.
\end{align}
For the first summand in (\ref{eq:crimes-10}), 
stability of the hypersingular operator $\hyp:H^{1/2}(\Gamma)\to H^{-1/2}(\Gamma)$, 
assumptions~\eqref{eq:cor:variational-crimes-100}--\eqref{eq:cor:variational-crimes-200}, 
and Lemma~\ref{lemma:a-priori-dual} yield
\begin{eqnarray*}
 |\varepsilon_1(w-w_h)|
 &=& |\langle \phi_0 - \Pi^{k-1}_h \phi_0, w - w_h\rangle + \langle \hyp( u_0 - \Pi^{k}_h u_0),w - w_h\rangle|
\\ 
 &\lesssim& \Big[\| \phi_0 - \Pi^{k-1}_h \phi_0\|_{H^{-1/2}(\Gamma)} + \|u_0 - \Pi^k_h u_0\|_{H^{1/2}(\Gamma)}
\Big] \|w - w_h\|_{H^{1/2}(\Gamma)} 
\\&\stackrel{\eqref{eq:cor:variational-crimes-100}-\eqref{eq:cor:variational-crimes-200}}\lesssim& 
\Big[h^{k+1/2}\|\phi_0\|_{H^k_{pw}(\Gamma)} + h^{k+1/2}\| u_0 \|_{H^{k+1}_{pw}(\Gamma)}\Big]
\, \|w - w_h\|_{H^{1}(\Omega)}
 \\&\stackrel{\rm Lem.~\ref{lemma:a-priori-dual}}\lesssim& h^{k+1/2}\Big[\|\phi_0\|_{H^k_{pw}(\Gamma)} + \| u_0 \|_{H^{k+1}_{pw}(\Gamma)}\Big]\,h\, \|\cutoff e\|_{L^2(\Omega)}.
\end{eqnarray*}
For the second summand in (\ref{eq:crimes-10}), we argue similarly, but rely on estimate~\eqref{eq:crimes3a} to see
\begin{eqnarray*}
 |\varepsilon_1(w-w_\varepsilon)|
&\stackrel{\eqref{eq:cor:variational-crimes-100}-\eqref{eq:cor:variational-crimes-200}}\lesssim&
\Big[h^{k+1/2}\|\phi_0\|_{H^k_{pw}(\Gamma)} + h^{k+1/2}\| u_0 \|_{H^{k+1}_{pw}(\Gamma)}\Big]
\, \|w - w_\varepsilon\|_{H^{1}(\Omega)}
\\&\stackrel{\eqref{eq:crimes3a}}\lesssim&
h^{k+1/2}\Big[\|\phi_0\|_{H^k_{pw}(\Gamma)} + \| u_0 \|_{H^{k+1}_{pw}(\Gamma)}\Big]\,h\, \|\cutoff e\|_{L^2(\Omega)}.
\end{eqnarray*}
For the third summand in (\ref{eq:crimes-10}), stability of $\hyp:L^2(\Gamma)\to H^{-1}(\Gamma)$ yields
\begin{eqnarray*}
 |\varepsilon_1(w_\varepsilon)|
 &=& |\langle \phi_0 - \Pi^{k-1}_h \phi_0, w_\varepsilon\rangle + \langle \hyp( u_0 - \Pi^{k}_h u_0),w_\varepsilon\rangle |\\
 &\le& \Big[\|\phi_0 - \Pi^{k-1}_h \phi_0\|_{H^{-1}(\Gamma)} + \|u_0 - \Pi^{k}_h u_0\|_{L^2(\Gamma)} \Big]\,
\|w_\varepsilon\|_{H^1(\Gamma)}
 \\&\stackrel{\eqref{eq:cor:variational-crimes-100}-\eqref{eq:cor:variational-crimes-200}}\lesssim&
\Big[h^{k+1}\|\phi_0\|_{H^k_{pw}(\Gamma)} + h^{k+1}\| u_0 \|_{H^{k+1}_{pw}(\Gamma)}\Big]\, \|w_\varepsilon\|_{H^1(\Gamma)}.
\end{eqnarray*}
For the control of $\|w_\varepsilon\|_{H^1(\Gamma)}$, we use 
\eqref{eq:crimes3b} and the continuity of the trace operator 
$\gamma: B^{1/2}_{2,1}(\Omega) \rightarrow L^2(\Gamma)$ 
(see, \cite[Thm.~{2.9.1}]{triebel95} for the present case of polygons/polyhedra or 
\cite[Lemma~{2.1}]{melenk-wohlmuth14a} for the case of Lipschitz domains)
to get 
\begin{equation*}
\|w_\varepsilon\|_{H^1(\Gamma)}
\le \|w_\varepsilon\|_{L^2(\Gamma)} + \|\nabla w_\varepsilon\|_{L^2(\Gamma)} 
\lesssim \|w_\varepsilon\|_{B^{1/2}_{2,1}(\Omega)}  + \|\nabla w_\varepsilon\|_{B^{1/2}_{2,1}(\Omega)} 
\lesssim \|w_\varepsilon\|_{B^{3/2}_{2,1}(\Omega)} 
\lesssim h^{1/2} \,|\ln h| \,\|\cutoff e\|_{L^2(\Omega)}.
\end{equation*}
Combining the last estimates, we arrive at
\begin{align*}
 |\varepsilon_1(w_h)|
 \lesssim h^{k+3/2}(1+|\ln h|) \,\Big[\|\phi_0\|_{H^k_{pw}(\Gamma)} + \| u_0 \|_{H^{k+1}_{pw}(\Gamma)}\Big]\, \|\cutoff e\|_{L^2(\Omega)}, 
\end{align*}
which is~\eqref{eq:crimes1}.

We now indicate the proof of \eqref{eq:crimes2}: From the {\sl a~priori} estimate~\eqref{eq:lemma:B32-regularity-2} of Lemma~\ref{lemma:B32-regularity},
we obtain $\|\lambda\|_{B^0_{2,\infty}(\Gamma)} \lesssim h^{1/2} \|\cutoff e\|_{L^2(\Omega)}$.
Lemma~\ref{lemma:B12infty-trace} with $\delta = 1/2$ and $\varepsilon = h$ provides
$\lambda_\varepsilon\in H^{1/2}(\Gamma)$ with
\begin{subequations}\label{eq:crimes4}
\begin{align}\label{eq:crimes4a}
 &\|\lambda-\lambda_\varepsilon\|_{H^{-1/2}(\Gamma)}
 + h^{1/2}\,\|\lambda-\lambda_\varepsilon\|_{B_{2,\infty}^0(\Gamma)}
 + h\,\|\lambda_\varepsilon\|_{H^{1/2}(\Gamma)}
 \lesssim h^{1/2}\,\|\lambda\|_{B_{2,\infty}^0(\Gamma)}
 \lesssim h\, \|\cutoff e\|_{L^2(\Omega)},\\
 \label{eq:crimes4b}
 &\|\lambda_\varepsilon\|_{L^2(\Gamma)} 
 \lesssim |\ln h|^{1/2}\,\|\lambda\|_{B_{2,\infty}^0(\Gamma)}
 \lesssim h^{1/2}\,|\ln h|^{1/2}\,\|\cutoff e\|_{L^2(\Omega)}. 
\end{align}
\end{subequations}
Linearity of $\varepsilon_2$ yields
\begin{align*}
 |\varepsilon_2(\lambda_h)|
 \le |\varepsilon_2(\lambda-\lambda_h)| + |\varepsilon_2(\lambda-\lambda_\varepsilon)|
 + |\varepsilon_2(\lambda_h)|. 
\end{align*}
We recall stability of $K:H^{1/2}(\Gamma)\to H^{1/2}(\Gamma)$ as well as $K:L^2(\Gamma)\to L^2(\Gamma)$.
We use this, assumption~\eqref{eq:cor:variational-crimes-100}, and estimates~\eqref{eq:crimes4}. 
Arguing as for $\varepsilon_1(w_h)$, we obtain
\begin{eqnarray*}
 |\varepsilon_2(\lambda_h)|
 &\stackrel{\eqref{eq:cor:variational-crimes-100}}\lesssim& h^{k+1/2}\| u_0 \|_{H^{k+1}_{pw}(\Gamma)}\,\Big[\|\lambda-\lambda_h\|_{H^{-1/2}(\Gamma)}
 + \|\lambda-\lambda_\varepsilon\|_{H^{-1/2}(\Gamma)}\Big]
 + h^{k+1}\| u_0 \|_{H^{k+1}_{pw}(\Gamma)}\,\|\lambda_\varepsilon\|_{L^2(\Gamma)}
\\&\stackrel{{\rm Lem.}~\ref{lemma:a-priori-dual} {\rm~and~} \eqref{eq:crimes4},}\lesssim&
h^{k+1/2}\,\| u_0 \|_{H^{k+1}_{pw}(\Gamma)}\,h\,\|\cutoff e\|_{L^2(\Omega)}
+ h^{k+1}\,\| u_0 \|_{H^{k+1}_{pw}(\Gamma)}\,h^{1/2}|\ln h|^{1/2}\,\|\cutoff e\|_{L^2(\Omega)}
\\&\le&h^{k+3/2}(1+|\ln h|^{1/2})\,\| u_0 \|_{H^{k+1}_{pw}(\Gamma)}\,\|\cutoff e\|_{L^2(\Omega)}.
\end{eqnarray*}
This proves~\eqref{eq:crimes2} and completes the proof.
\end{proof}

\begin{remark}{\normalfont
\begin{enumerate}
\item 
The estimate \eqref{eq:cor:variational-crimes-100} can be realized by standard (nodal) interpolation. 
The estimate \eqref{eq:cor:variational-crimes-200} can be achieved by $L^2(\Gamma)$-projection. Given 
that the space $M_h$ consists of discontinuous functions, this is again a local computation. Hence, 
Theorem~\ref{thm:variational-crimes} shows that interpolating 
$u_0$ does lead to a method that, up to a logarithmic factor,  preserves 
the convergence rates of Theorem \ref{thm:approximation-of-varphi}.
\item 
The proof of Theorem~\ref{thm:variational-crimes} shows that the logarithmic term can be removed 
if more regularity is required of $u_0$ and $\phi_0$ and correspondingly higher order interpolants
are employed. 
\eremk
\end{enumerate}
}
\end{remark}
\section{Numerical results.}
\label{sec:numerics}
%
\begin{figure}[t]
\psfrag{eps}{}
\centering
\includegraphics[width=.3\textwidth]{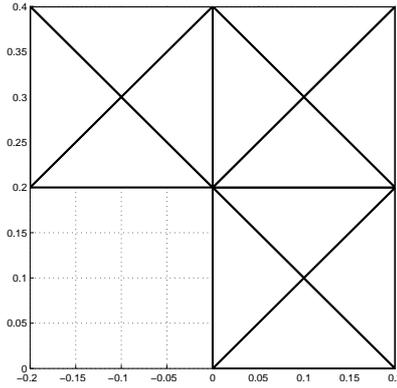}
\caption{Domain $\Omega\subset\mathbb R^2$ and initial triangulation
into $12$ triangles and $8$ boundary segments for the numerical
experiments in Section~\ref{sec:numerics}.}
\label{fig:lshape}
\end{figure}
\begin{figure}[t]
\psfrag{O(1)}{$O(h)$}
\psfrag{O(3/2)}{$O(h^{3/2})$}
\psfrag{O(2)}{$O(h^2)$}
\psfrag{O(5/2)}{$O(h^{5/2})$}
\psfrag{error(u) in H1}{$\norm{\nabla(u-u_h)}{L^2(\Omega)}$}
\psfrag{error(u) in L2}{$\norm{u-u_h}{L^2(\Omega)}$}
\psfrag{error(u) in L2 along Gamma}{$\norm{u-u_h}{L^2(S_h)}$}
\psfrag{error(phi) in L2}{$\norm{h^{1/2}(\varphi-\varphi_h)}{L^2(\Gamma)}$}
\psfrag{error}{error}
\psfrag{1/h}{$1/h$}
\centering
\includegraphics[width=.7\textwidth]{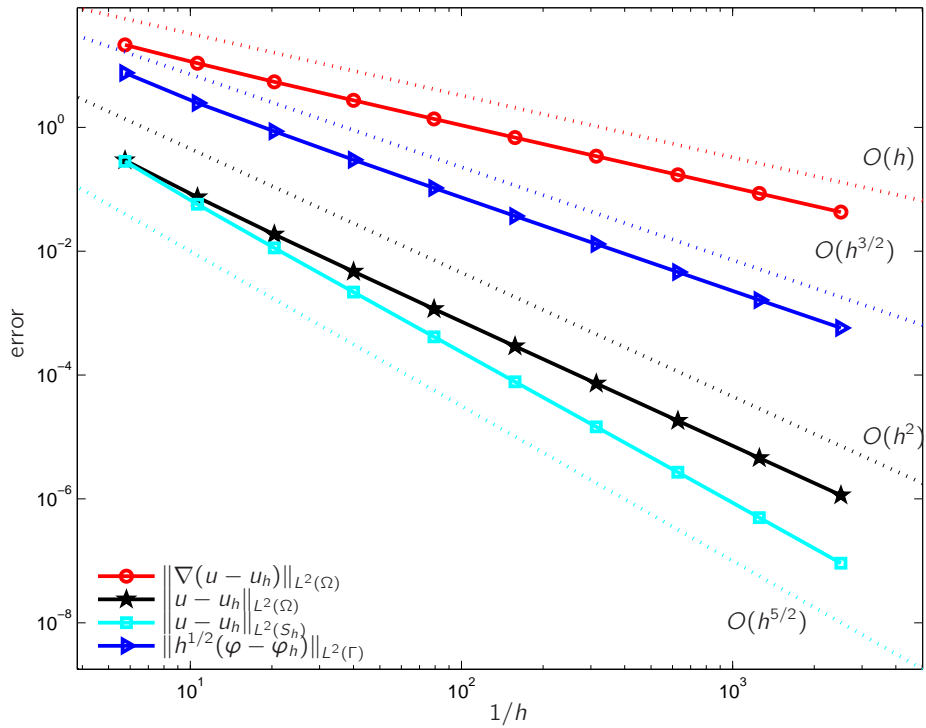}
\caption{Performance of lowest-order FEM-BEM ($k=1$) with ${\alpha = 3/2}$ in~\eqref{eq:ex:solution} for refinement levels $\ell=0,\dots,9$.}
\label{fig:lshape:p0s1}
\end{figure}
\begin{figure}[t]
\psfrag{O(3)}{$O(h^3)$}
\psfrag{O(7/2)}{$O(h^{7/2})$}
\psfrag{O(2)}{$O(h^2)$}
\psfrag{O(5/2)}{$O(h^{5/2})$}
\psfrag{error(u) in H1}{$\norm{\nabla(u-u_h)}{L^2(\Omega)}$}
\psfrag{error(u) in L2}{$\norm{u-u_h}{L^2(\Omega)}$}
\psfrag{error(u) in L2 along Gamma}{$\norm{u-u_h}{L^2(S_h)}$}
\psfrag{error(phi) in L2}{$\norm{h^{1/2}(\varphi-\varphi_h)}{L^2(\Gamma)}$}
\psfrag{error}{error}
\psfrag{1/h}{$1/h$}
\centering
\includegraphics[width=.7\textwidth]{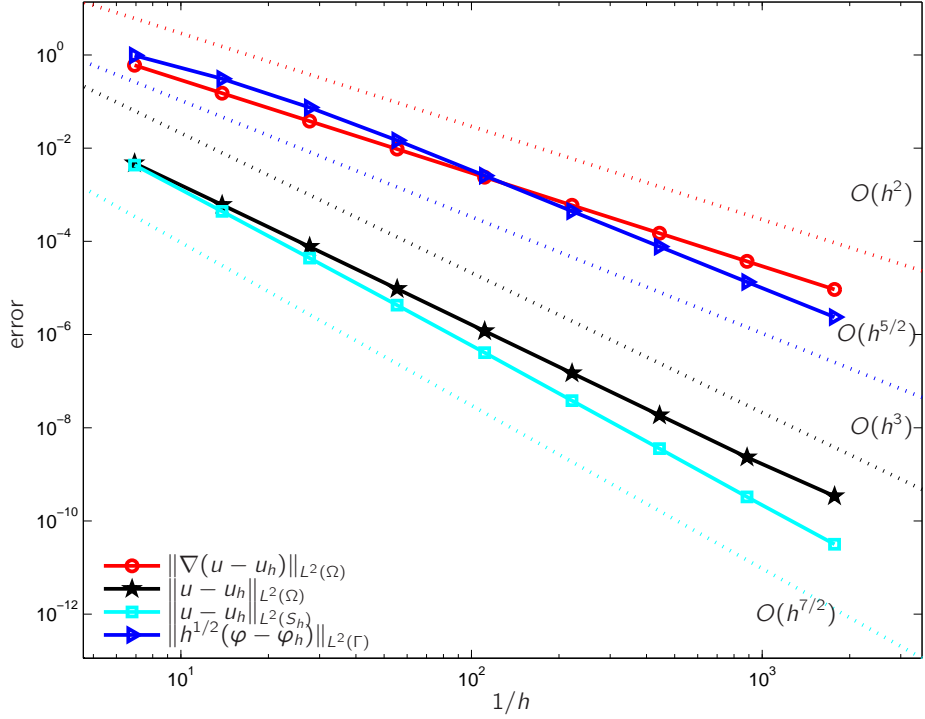}
\caption{Performance of higher-order FEM-BEM ($k=2$) with ${\alpha = 5/2}$ in~\eqref{eq:ex:solution} for refinement levels $\ell=0,\dots,8$.}
\label{fig:lshape:p1s2}
\end{figure}

This section underlines the theoretical results of 
Theorem~\ref{thm:approximation-of-varphi} and 
Theorem~\ref{thm:variational-crimes}. Throughout, we consider
the L-shaped domain 
$$\Omega = (-0.2,0.2)\times (0,0.4)\setminus [-0.2,0]\times[0,0.2]$$
visualized in Figure~\ref{fig:lshape}. For an $\alpha>0$, we prescribe the exact solution
$(u,u^\extr)$ of~\eqref{eq:strongform} as 
\begin{subequations}
\label{eq:ex:solution}
\begin{align}
 u(x,y) &= 1000\cdot{\rm Re}\big(z^\alpha\big)
 &&\text{with } z = x+\bi y,\\
 u^\extr(x,y) &= {\rm Re}\,\big(1/(z-v)\big)
 &&\text{with }z=x+\bi y\text{ and }v = 0.1 (1 + \bi),
\end{align}
\end{subequations}
where $\opA$ is the identity, i.e., $-\nabla\cdot(\opA\nabla u)=-\Delta u$.
We note that $f:=-\Delta u = 0$ in $\Omega$ and $-\Delta u^\extr = 0$ 
in $\mathbb R\backslash\overline\Omega$. The data $f$, $u_0$, and $\phi_0$
in~\eqref{eq:strongform} are calculated from the prescribed exact solutions.
The exterior solution $u^\extr$ is smooth in $\mathbb R\backslash\overline\Omega$,
while $u$ has a singularity at the lower left corner $(0,0)$ 
of $\Omega$, whose strength is controlled by $\alpha$. Away from this singularity, 
$u$ is smooth, in particular, near 
the reentrant corner $(0,0.2)$ of $\Omega$, where the solutions of the dual
problem~\eqref{eq:dual-problem} and the bidual problem~\eqref{eq:bidual-problem} have 
a singularity.

With $\varphi=\partial_n^\extr u^\extr$, the pair $(u,\varphi)$
is the unique solution of~\eqref{eq:rhs-weak}--\eqref{eq:weak-form}. By our 
choice of $u^\extr$, the function $\varphi$ is edgewise smooth and satisfies, 
in particular, all regularity requirements of the present work. 

\begin{remark}
\label{remk:regularity-of-solution-numerical-example}
\begin{enumerate}
\item 
For $\alpha \not\in\BbbN$, the solution $u$ in $\Omega$ has the regularity
$u \in B^{1+\alpha}_{2,\infty}(\Omega)$. We will use $\alpha = 3/2$ for the case
$k=1$ and $\alpha = 5/2$ with $k = 2$; that is, $k +3/2 = 1 +\alpha$ and
$u \in B^{k+3/2}_{2,\infty}(\Omega)$. 
This regularity is marginally lower than what is required in 
Theorem~\ref{thm:approximation-of-varphi}. Nevertheless, up to logarithmic terms   
(cf.~Lemma~\ref{lemma:B12infty-trace} and the proof of 
Theorem~\ref{thm:variational-crimes}) we expect the results 
of Theorem~\ref{thm:approximation-of-varphi} to hold. 
\item 
The scaling factor $1000$ in our definition of the solution $u$ is chosen to 
ensure that the approximation of $u$ dominates also {\em preasymptotically} 
in the standard {\sl a priori} estimate 
$$
\|u - u_h\|_{H^1(\Omega)} + \|\varphi - \varphi_h\|_{H^{-1/2}(\Gamma)} 
\lesssim \inf_{v \in V_h} \|u - v\|_{H^1(\Omega)} + \inf_{\mu \in M_h} \|\varphi - \mu\|_{H^{-1/2}(\Gamma)}, 
$$
in particular for the case $k = 2$. A calculation gives: 
\begin{align*}
\mbox{ $\alpha = 5/2$:} & \qquad 
|u|_{H^3(\Omega)} \approx 2900, \qquad 
|\varphi|_{H^2(\Gamma)} \approx 867, \\
\mbox{ $\alpha = 3/2$:} & \qquad 
|u|_{H^2(\Omega)} \approx 828, \qquad 
|\varphi|_{H^1(\Gamma)} \approx 35. 
\end{align*}
\eremk
\end{enumerate}
\end{remark}

Throughout, we consider a sequence of triangulations 
$\TT_\Omega$ that are obtained by uniform red refinement of the initial 
triangulation $\TT_\Omega$ depicted in Figure~\ref{fig:lshape}. The boundary mesh
$\TT_\Gamma$ is always induced by the volume mesh 
$\TT_\Gamma = \TT_\Omega|_\Gamma$. The computations consider
meshes with $\#\TT_\Omega = 12\cdot4^\ell$ triangles and
$\#\TT_\Gamma = 8\cdot2^\ell$ boundary segments, i.e., 
mesh size $h = 0.2\cdot2^{-\ell}$ for refinement levels $\ell=0$, $1$, $2,\dots$.
Our computations are performed in {\sc MATLAB} by means of the BEM library
HILBERT~\cite{aurada-ebner-feischl-ferraz-leite-fuehrer-goldenits-karkulik-mayr-praetorius14}. 
The linear systems are solved with the {\sc MATLAB} backslash operator.
The matrix entries, in particular of the BEM matrices, are computed analytically. 

In our numerical experiments,
we let $(u_h,\varphi_h)\in X_h = \mathcal S^{k,1}(\mathcal T_\Omega)\times\mathcal S^{k-1,0}(\mathcal T_\Gamma)$ 
be the Galerkin solution of
\begin{subequations}\label{eq:81}
\begin{eqnarray}\label{eq:81a}
\widetilde a(u,v) - b(v,\varphi) &=& \langle \phi_0 + \hyp \Pi_hu_0,v\rangle_\Gamma \qquad \forall v \in V_h, \\
b(u,\psi) + c(\varphi,\psi) &=& \langle \psi,(1/2-\dlo)\Pi_hu_0\rangle_\Gamma) \qquad \forall \psi \in M_h.
\end{eqnarray}
\end{subequations}
Here, $\Pi_h:L^2(\Gamma)\to\mathcal S^{k,1}(\mathcal T_\Gamma)$ denotes the
$L^2$-orthogonal projection. We note that $\mathcal S^{k,1}(\mathcal T_\Gamma)
= \big\{v|_\Gamma\,:\,v\in\mathcal S^{k,1}(\mathcal T_\Omega)\big\}$ is 
the discrete trace space and that this choice satisfies the assumption~\eqref{eq:cor:variational-crimes-100}
of Theorem~\ref{thm:variational-crimes} provided that $u_0 = u-u^\extr$ and
$\phi_0 = (\nabla u - \nabla u^\extr)\cdot n$ are sufficiently
smooth, i.e., $\alpha >  k+1/2$ in~\eqref{eq:ex:solution}. Our actual choice $\alpha = k+1/2$ 
corresponds to a limiting case, for which we still expect the convergence results to hold, 
up to logarithmic terms (cf. also Remark~\ref{remk:regularity-of-solution-numerical-example}).
The term $\langle \phi_0,v\rangle$ in 
(\ref{eq:81a}) is treated by a high order Gaussian quadrature rule. 


In our experiments, we consider the lowest-order case $k=1$ as well as $k=2$
and plot the errors 
\begin{itemize}
\item $\norm{u-u_h}{L^2(\Omega)}$,
\item $\norm{\nabla(u-u_h)}{L^2(\Omega)}$,
\item $\norm{h^{1/2}(\varphi-\varphi_h)}{L^2(\Gamma)}$, 
\item $\norm{u-u_h}{L^2(S_h)}$, where 
$S_h := \bigcup\big\{T\in\TT_\Omega\,:\,\overline T\cap\Gamma\neq\emptyset\big\}$,
\end{itemize}
versus the mesh size $1/h$.
Lemma~\ref{lemma:standard-variational-crimes}  provides 
$\norm{\nabla(u-u_h)}{L^2(\Omega)} = O(h^k)$ and 
Theorem~\ref{thm:variational-crimes} guarantees 
(up to logarithmic terms) 
$\norm{h^{1/2}(\varphi-\varphi_h)}{L^2(\Gamma)}
= O(h^{k+1/2})$ as well as 
$\norm{u-u_h}{L^2(S_h)} = O(h^{k+3/2})$. These rates are 
observed numerically in Figures~\ref{fig:lshape:p0s1} and \ref{fig:lshape:p1s2}
for the cases $k=1$, $2$, respectively.
%


\acks
The first two authors acknowledge support by the Austrian Science Fund (FWF) under grants 
W1245 (doctoral program ``dissipation and dispersion in nonlinear PDEs''; JMM, DP) and 
P21732 (project ``Adaptive Boundary Element Methods'', DP). 
\bibliography{nummech,dpr}
\bibliographystyle{plain}
\end{document}